\documentclass{amsart}
\usepackage{fullpage,graphicx,amsfonts,amssymb,amsmath,amsthm}
\usepackage{xcolor}
\usepackage{bbm}
\usepackage[all]{xy}
\usepackage[inline]{enumitem}
\usepackage{hyperref}
\usepackage{arydshln}
\usepackage{mathrsfs}
\usepackage{MnSymbol}

\newtheorem{thm}{Theorem}
\newtheorem{theorem}[thm]{Theorem}

\newtheorem{proposition}[thm]{Proposition}
\newtheorem{lemma}[thm]{Lemma}
\newtheorem{corollary}[thm]{Corollary}

\theoremstyle{definition}

\newtheorem{definition}[thm]{Definition}
\newtheorem{remark}[thm]{Remark}

\newcommand{\A}{\mathbb{A}}
\newcommand{\C}{\mathbb{C}}

\newcommand{\Q}{\mathbb{Q}}
\newcommand{\R}{\mathbb{R}}

\newcommand{\Z}{\mathbb{Z}}

\newcommand{\cH}{\mathcal{H}}

\newcommand{\cL}{\mathcal{L}}

\newcommand{\cO}{\mathcal{O}}

\newcommand{\fm}{\mathfrak{m}}
\newcommand{\fp}{\mathfrak{p}}

\newcommand{\Hom}{\operatorname{Hom}}
\newcommand{\rk}{\operatorname{rk}}

\newcommand{\Sym}{\operatorname{Sym}}
\newcommand{\supp}{\operatorname{supp}}
\renewcommand{\ss}{\operatorname{ss}}
\newcommand{\Gal}{\operatorname{Gal}}
\newcommand{\cyc}{\operatorname{cyc}}

\newcommand{\St}{\operatorname{St}}

\newcommand{\tr}{\operatorname{tr}}
\newcommand{\Ind}{\operatorname{Ind}}
\newcommand{\GL}{\operatorname{GL}}
\renewcommand{\Re}{\operatorname{Re}}
\renewcommand{\Im}{\operatorname{Im}}
\newcommand{\AG}{\operatorname{AG}}

\newcommand{\sw}{{\sf{w}}}

\newenvironment{spmatrix}{\left(\begin{smallmatrix}}{\end{smallmatrix}\right)}

\allowdisplaybreaks

\title{Parahoric level $p$-adic $L$-functions for automorphic representations of $\GL_{2n}$ with Shalika models}
\author{Mladen Dimitrov \and Andrei Jorza}
\date{}

\begin{document}

\maketitle

\begin{abstract}
We construct $p$-adic $L$-functions for regularly refined cuspidal automorphic representations  of symplectic type on $\GL_{2n}$ over totally real fields,
which are parahoric spherical at every finite place. Furthermore, we prove etaleness of the parabolic eigenvariety at such points and
construct  $p$-adic $L$-functions in families. The novel local ingredients are the construction of improved Ash--Ginzburg Shalika functionals
and production of Friedberg--Jacquet test vectors relating local zeta integrals to automorphic $L$-functions beyond the spherical level. 
Our proofs rely on a generalization of Shahidi's theory of local coefficients to Shalika models, for which we establish a general factorization formula related to the exterior square automorphic $L$-function. 
\end{abstract}

\addtocontents{toc}{\setcounter{tocdepth}{0}}

\section*{Introduction}

Two of the most stunning recent advances towards the Birch--Swinnerton-Dyer conjecture are  Skinner's  converse theorem  \cite{skinner:bsd}  
for certain arithmetic rank $1$ elliptic curves and Loeffler--Zerbes' work \cite{loeffler-zerbes:bsd} for certain analytic rank $0$
 abelian surfaces. In both cases, $p$-adic $L$-functions provide the crucial connection between the analytic and algebraic sides.

The present paper is part of a program aiming  to attach $p$-adic $L$-functions to cuspidal automorphic representations $\pi$ of $\GL_N$ over a totally real number field and study their properties, such as trivial zeros. We are focusing our study to the already large class of $\pi$ which are regular algebraic (i.e. cohomological of weight $\lambda$) and essentially self-dual of symplectic type (i.e. admitting the so-called Shalika model) forcing $N=2n$ to be even.  In contrast to constructions based on the Rankin--Selberg method (e.g. Januszewski, Namikawa,  Eischen--Harris--Li--Skinner for unitary groups,  or Wan, Eischen--Wan, Rosso--Liu for symplectic groups) our approach  does not  require $\pi_p$ to be ordinary, but only  of finite slope. 

In this paper,  we provide the first  construction of $p$-adic $L$-functions and $p$-adic families, for finite slope cuspidal automorphic representations of $\GL_{N}$ 
allowing ramification at $p$. More precisely we allow parahoric level both at $p$ and in the tame level (with respect to the Siegel parabolic), which is the natural square-free level for studying $p$-adic $L$-functions of representations with Shalika models (see \cite[\S1.2.3]{BDW} for a thorough discussion and numerical examples). To achieve this, we develop a theory of Shalika local coefficients analogous to Shahidi's classical theory of local
coefficients \cite{shahidi:L}, and prove a consequential factorization result that allows us to construct an improved Ash--Ginzburg Shalika functional  for parahoric spherical local representations. In passing, we show existence of explicit Friedberg--Jacquet  test vectors (previously only available in the spherical case) allowing us to  
computing the local zeta integrals at all finite order characters. 
Using the results of \cite{BDW}, this informs the geometry of the parabolic eigenvariety at new points and allows to construct 
a two-variable $p$-adic $L$-functions with controlled growth and satisfying the precise interpolation relations conjectured by  Coates and Perrin-Riou (formulated originally  in \cite{coates:motivic-Lp} for good reduction at $p$, but later observed by Hansen that the statement  makes sense more generally).

\begin{theorem}\label{t:main}
Let $\pi$ be a RASCAR on $\GL_{2n}$ over a totally real field $F$ which is parahoric spherical at all finite places, 
and let  $\widetilde{\pi}=(\pi, (\alpha_\fp)_{\fp\mid p})$ be regular parahoric refinement which is non-critical. 
There exists an admissible distribution $\mathcal{L}_p(\widetilde{\pi})$ on $\Gal_p$ (the Galois group of the maximal abelian extension of $F$ unramified outside $p\infty$) of controlled growth such that 
for every finite order Hecke character $\chi$ of $\Gal_p$ and all $j \in \mathrm{Crit}(\lambda)$
	\begin{align}\label{eq:interpolation}
		\iota_p^{-1}(\cL_p(\tilde\pi, \chi\chi_{\cyc}^j)) =  \mathcal{G}(\chi_f)^n  \mathrm{N}_{F/\Q}(-i\mathfrak{d})^{jn} e_p(\tilde\pi_p,\chi,j)
\cdot \frac{L^{(p)}\big(\pi\otimes\chi, j+\tfrac{1}{2}\big)}{\Omega_{\tilde\pi}^{\epsilon}},
	\end{align}
	where  $\epsilon = (\chi\chi_{\cyc}^j\eta)_\infty$,  $\mathcal{G}(\chi_f)$ is the Gauss sum and  $e_p(\tilde\pi_p,\chi,j)$
	is  a local  factor (see Theorem~\ref{thm:non-ordinary}) 
\end{theorem}

The use of the Friedberg--Jacquet integral representation of the standard $L$-function of $\pi$ for 
 $p$-adic interpolation of the critical values  was initiated  in Ash--Ginzburg \cite{AG}  and  further developed  in \cite{DJR}.  
 The works  \cite{BDW,BDGJW} provide a fairly complete construction of  $p$-adic $L$-functions in families based on
 Stevens' theory of overconvergent cohomology and  use them to prove profound results on the geometry of the $\GL_{2n}$ eigenvariety,
 including  étalness over the pure weight space and density of classical points of Shalika type.  
Previous results, however, crucially relied  on the existence of   explicit Friedberg--Jacquet test vectors, limiting most of the
  results to $\pi$ everywhere unramified. 

In this paper, we succeed in constructing $p$-adic $L$-functions beyond level 1 by generalizing Shahidi's work on local coefficients, which arose in his construction of general automorphic $L$-function \cite{shahidi:L}, to the setting of Shalika models. Shahidi's local coefficients appear as Whittaker model multiplication factors in the composition of simple intertwining operators, which provide the necessary Euler factors of automorphic $L$-functions. In the case of automorphic representations with Shalika models, our local coefficients generate the local Euler factors of the exterior square automorphic $L$-function. This factorization result, which is the main result of \S \ref{s:intertwining}, corroborates the relationship between exterior square automorphic $L$-functions and Shalika models.

In appearance of  technical nature, the improvements of Theorem~\ref{t:main} are crucial for the arithmetic applications we aim at, namely the 
Greenberg--Benois trivial zero conjecture at the central point for $L_p(\tilde\pi,s)$. 
 Indeed, the occurrence of a central trivial zero forces $\pi_p$ to ramify and the conjectural formula  requires evaluations at the trivial character. 
Allowing  parahoric tame level also significantly broadens the scope of application of our results, the previous construction of $p$-adic $L$-functions in 
families being unconditional  only  for level $1$. 

\begin{remark}
In order to obtain precise interpolation formulas à la Coates--Perrin-Riou, we must start with the Ash--Ginzburg Shalika functional on the full induced representation, which is only defined if the character induced is regular and ordered in a particular way. It is this requirement that forces us to restrict our attention to regular parahoric spherical representations: they are the only ones that can be realized as subrepresentations of such Ash--Ginzburg ordered characters.
\end{remark}

{ \noindent{\it Acknolwedgements:} { \small
We would like to thank  Dinakar Ramakrishnan and  Chris Williams for helpful comments and for many stimulating conversations. 
The authors acknowledge the support of the CDP C${}^2$EMPI, as well as the French State under the France-2030 programme, the University of Lille, the Initiative of Excellence of the University of Lille, the European Metropolis of Lille for their funding and support of the R-CDP-24-004-C2EMPI project.
The first named author was partially supported by  the Agence Nationale de la Recherche grant ANR-18-CE40-0029. 
} 

\tableofcontents

\addtocontents{toc}{\setcounter{tocdepth}{1}}

\section{Parahoric spherical representations admitting a Shalika model}\label{s1}
In this section, we specify the relationship between parahoric invariants and Shalika models for generic representations.

Let $F/\mathbb{Q}_p$ be a finite extension with ring of integers $\mathcal{O}$,  uniformizer $\varpi$ and  different ideal $(\varpi^{\delta})$. 
Let $B$ be the upper triangular Borel subgroup of  $G=\GL_{2n}(F)$. The Iwahori subgroup $I\subset
\GL_{2n}(\mathcal{O})$ consists of matrices which are upper triangular mod $\varpi$.
Let $M_n=M_n(F)$ and  $B_n$ the Borel subgroup of $G_n=\GL_n(F)$.  

We fix an additive character $\psi:F\to \mathbb{C}^\times$ with kernel $(\varpi^{-\delta})$ and we recall that the Shalika subgroup 
$S=G_n\rtimes M_n$ of $G$ is  endowed with a character $\psi$ sending $s=\begin{spmatrix} h&\\&h\end{spmatrix}\begin{spmatrix} 1&X\\&1\end{spmatrix}\in S$ to $\psi(\tr(X))$. Note that as the object of our study is parahoric spherical representations, we can restrict our local study to the case $\eta=\mathbf{1}$. 

We say that an an irreducible smooth representation $\pi$ of $G$ admits a Shalika model if can be embedded $G$-equivariently in $\Ind_S^G\psi$, or equivalently by Frobenius reciprocity, if there exists a non-zero  map $\mathcal{S}_\pi:\pi \to \mathbb{C}$ such that
$\mathcal{S}_\pi(s\cdot f)=\psi(s)\mathcal{S}_\pi(f)$ for all $s\in S$. Note that by the Multiplicity One Theorem for Shalika models \cite[Prop. 6.1]{jacquet-rallis}, if such a functional exists then it  is unique, up to a non-zero scalar. Given a  Shalika functional $\mathcal{S}_\pi$, the Shalika function attached to  $f \in \pi$ is
\[W_f:G\to \mathbb{C}, \quad g\mapsto \mathcal{S}_\pi(g\cdot f), \] 
and $f\mapsto W_f$ is the  embedding of $\pi$ in $\Ind_S^G\psi$. 
The work of Friedberg--Jacquet \cite{friedberg-jacquet} gives an integral representation for the  $L$-function of $\pi$ in terms of its Shalika model. 
Namely,  for  $W\in \Ind_S^G\psi$, the zeta integral 
\[\zeta(\chi,W,s)=\int_{G_n}W(\begin{spmatrix} h&\\& \mathbf{1}\end{spmatrix})\chi(\det h)|\det h|^{s-1/2}dh\]
converges absolutely for $\Re(s)\gg 0$  and there exists a test vector $W_\pi$ such that $\zeta(W_\pi,\chi,s)=L(\pi\otimes\chi,s)$.

A result of Matringe \cite[Cor.~1.1]{matringe:shalika} describes all generic representations $\pi$ admitting Shalika models. For instance, when $\pi$ is Iwahori spherical and generic, $\pi$ has a Shalika model  if, and only if, $\pi$ is isomorphic to the normalized parabolic induction
\begin{equation}\label{eq:matringe}
\Ind_P^G\left( \prod_i (\theta_i\times
\theta_i^{-1})\times\prod\limits_{m=2}^{n}\prod\limits_{j} (\eta_{j,m}\otimes\St_m\times\eta_{j,m}^{-1}\otimes\St_m ) \times\prod\limits_{m=1}^{n}\prod\limits_{j} \varepsilon_{j,m}\otimes\St_{2m}\right)
\end{equation}
where the $\eta_{j,m}$'s are arbitrary characters, while  the $\varepsilon_{j,m}$'s  are trivial or quadratic, and where 
$\St_m$ is the $\GL_m$-Steinberg and $P$ is an appropriate parabolic so that the induction makes sense. We will see in Proposition~\ref{p:parahoric-rank2}  that   $\pi$ is parahoric spherical if and only if the only Steinberg representations occurring in \eqref{eq:matringe} are $\St_2$'s.

We now introduce  a  class of parabolic subgroups of Levi subgroups  of $G$. Let $\Delta$ be the positive roots of $G$
corresponding to $B$. If $T\subset S\subset \Delta$, we denote
$P_{S,T}$ the subgroup of $G$ with Lie algebra
\[\mathfrak{p}_{S,T}=\mathfrak{t} \oplus \bigoplus_{\alpha\in
  \mathbb{Z}_{+}S-\mathbb{Z}_{+}T}\mathfrak{g}_{\alpha}.\]
We remark that $P_{\Delta,\emptyset}=B$,
$P_{\Delta,\Delta}=G$, while $P_{\Delta,S}$ is a usual
parabolic subgroup, with Levi equal to $P_{S,S}$. Let 
\[J_{S,T}=\{g\in P_{S,S}(\mathcal{O})\mid g\mod\varpi\in
P_{S,T}(\mathcal{O}/\varpi)\},\]
which specializes to $J_{\Delta,\emptyset}=I$ and
$J_{\Delta,\Delta}=K$. Let $W$ denote the Weyl group of $G$. 

For a subset $S\subset \Delta$, we denote $W_S$ the Weyl group
of $P_{S,S}$. For  $S,T\subset \Delta$ by \cite[Prop 1.3.1]{casselman:book} we have 
\[P_{\Delta,S}\backslash G/P_{\Delta,T}\cong W_S\backslash
W/W_T,\]
the latter having representatives $\{w\in W\mid w T >0,
w^{-1}S>0\}$. It  follows  from the Iwahori decomposition that
\[P_{\Delta,S}\backslash G/J_{\Delta,T}\cong W_S\backslash
W/W_T.\]
\begin{lemma}\label{l:levi-parahoric}
Let $w\in W$, chosen such that $w T>0$ and $w^{-1}S>0$. Then
\[\operatorname{Levi}(w^{-1}P_{\Delta,S}w)\cap J_{\Delta,T}=J_{w^{-1}S,T\cap w^{-1}S}.\]
\end{lemma}
\begin{proof}
First, $w^{-1}P_{\Delta,S}w=P_{w^{-1}\Delta,w^{-1}S}$. If $g\in
w^{-1}P_{\Delta,S}w\cap J_{\Delta,T}$, then  $(g\mod\varpi)\in
P_{w^{-1}\Delta,w^{-1}S}\cap P_{\Delta,T}$. As 
\begin{align*}
\mathfrak{p}_{w^{-1}\Delta,w^{-1}S}=\mathfrak{t}\oplus\bigoplus_{\alpha\in
                                     \mathbb{Z}_{+}w^{-1}\Delta-\mathbb{Z}_{+}
                                     w^{-1}S}\mathfrak{g}_\alpha \quad \text{  and } \quad
  \mathfrak{p}_{\Delta,T}=\mathfrak{t}\oplus\bigoplus_{\alpha\in \mathbb{Z}_{+}\Delta-\mathbb{Z}_{+}T}\mathfrak{g}_\alpha 
\end{align*}
their  intersection equals  $\mathfrak{t} \oplus\bigoplus\mathfrak{g}_\alpha$, where the direct sum is taken over 
$\alpha\in (\mathbb{Z}_{+}w^{-1}\Delta-\mathbb{Z}_{+}w^{-1}S)\cap (\mathbb{Z}_{+}\Delta-\mathbb{Z}_{+}T)$.

Any $\alpha$ in this intersection can be written in an irreducible way as
\[\alpha=\alpha_{+}-\alpha_T=w^{-1}\beta_{+}-w^{-1}\beta_S,\]
where $\alpha_{+},\beta_{+}\in \mathbb{Z}_{+}\Delta$, $\alpha_T\in \mathbb{Z}_{+}T$ and  $\beta_S\in \mathbb{Z}_{+}S$.  
As $w^{-1}S>0$ we deduce  $ w^{-1}\beta_S\leqslant \alpha_T$. Applying $w$ we also get
\[w\alpha_+-w\alpha_T=\beta_+-\beta_S,\]
and this time $w T>0$ implies that $\alpha_T \leqslant w \beta_S$. 
We  deduce that $\alpha_+=w^{-1} \beta_+\in \Delta\cap w^{-1}\Delta$ and $\alpha_T=w^{-1} \beta_S \in T\cap w^{-1}S$,  and so
\[\mathfrak{p}_{w^{-1}\Delta,w^{-1}S}\cap \mathfrak{p}_{\Delta,T}=\mathfrak{p}_{\Delta\cap
  w^{-1}\Delta,T\cap w^{-1}S}.\]
Further restricting to the Levi subgroup implies the desired identity of groups. 
\end{proof}

\begin{lemma}\label{l:parahoric-condition}
Assume  $\pi$ as in \eqref{eq:matringe} with $P=P_{\Delta,S}$. 
If $\pi$ has non-zero invariants by  $J_{\Delta,T}$, for some $T$, then there exists   $w\in W$ such that $w T>0$,  $w^{-1}S>0$ and
$T\cap w^{-1}S=\emptyset$.  
\end{lemma}
\begin{proof}
By Mackey's formula, Frobenius reciprocity and Lemma \ref{l:levi-parahoric}, we have 
\begin{align*}
\pi^{J_{\Delta,T}}&=\Hom_{J_{\Delta,T}}(\mathbf{1}, \Ind_{P_{\Delta,S}}^G  \sigma)=
\bigoplus_{[w]\in P_{\Delta,S}\backslash G/J_{\Delta,T}} \Hom_{J_{\Delta,T}}(\mathbf{1}, \Ind_{(w^{-1}P_{\Delta,S}w)\cap J_{\Delta,T}}^{J_{\Delta,T}}\sigma^w)\\
  &=\bigoplus_{[w]\in W_S\backslash    W/W_T}(\sigma^w)^{\operatorname{Levi}(w^{-1}P_{\Delta,S}w)\cap
    J_{\Delta,T}}=\bigoplus_{[w]\in W_S\backslash    W/W_T}(\sigma^w)^{J_{w^{-1}S,T\cap w^{-1}S}}\cong\bigoplus_{[w]\in W_S\backslash
    W/W_T}\sigma^{J_{S,w T\cap S}}
\end{align*}
It follows that $\sigma^{J_{S,w T\cap S}}\neq 0$ for some $w\in W$ such that $w T >0$ and 
$w^{-1}S>0$. As the  Levi of $P_{\Delta,S}$ on which the representation $\sigma= \prod_i (\theta_i\otimes
\theta_i^{-1})\times \prod_j \sigma_j$ is defined is $P_{S,S}= \prod_i \GL(1)\times\GL(1)\times \prod_j \GL(m_j)$, we have 
\[\sigma^{J_{S,w T\cap S}}= \sigma^{P_{S,S}\cap J_{S,w T\cap S}}.\]
 The Iwahori  $J_{S,\emptyset}$ being the largest subgroup of $P_{S,S}$  under which $\sigma$ has non-zero invariants, it follows that  
 $P_{S,S}\cap J_{S, w T\cap S}\subset J_{S,\emptyset}$. Hence  $w T\cap S=\emptyset$ as desired.
\end{proof}

\begin{corollary}\label{c:parahoric-monodromy-rank}
Assume $\pi$ is as in \eqref{eq:matringe}. If $\pi^{J_{\Delta,T}}\neq 0$, where $P_{\Delta,T}$ has $m$ Levi blocks. Then each special representation $\sigma_i$ has Zelevinsky length $\leqslant m$, in other words it has monodromy rank $\leqslant m-1$.
\end{corollary}
\begin{proof}
By a block of a subset of $\Delta$ we refer to the roots of a Levi block of the associated parabolic.

By Lemma \ref{l:parahoric-condition}, for some $w$ we have $w^{-1}S\cap T=\emptyset$. Since $w$ satisfies $w
T>0,w^{-1}S>0$, this is equivalent to the condition that each
block of $T$ is disjoint from each block of $w^{-1}S$. Let
$\mathcal{B}$ be the largest block of $w^{-1}S$, of dimension
$d$, equal to the length of the largest Zelevinsky segment. We
will verify that $d\leqslant m$.

Writing $e_1,\ldots,e_{2n}$ for the standard characters of $T$,
the block $\mathcal{B}$ is generated by the positive roots
\begin{align*}
  \alpha_1&=e_{i_1}-e_{i_2}\\
  \alpha_2&=e_{i_2}-e_{i_3}\\
          &\vdots\\
  \alpha_{d-1}&=e_{i_{d-1}}-e_{i_d},
\end{align*}
in which case the positive roots of $\mathcal{B}$ are $\{e_{i_j}-e_{i_k}\mid j>k\}$. If $d>m$, by the pigeonhole principle there must exist some indices $j>k$ such that $i_j$ and $i_k$ lie in the same Levi block of $T$, in other words $e_{i_j}-e_{i_k}$ is also in a block of $T$, yielding a contradiction.
\end{proof}

As the Siegel parabolic subgroup has two Levi components Corollary \ref{c:parahoric-monodromy-rank} has the following consequence. 

\begin{proposition}\label{p:parahoric-rank2}
Let $\pi$ be an generic  smooth irreducible representation of $G$ admitting a Shalika model. Then $\pi$ is parahoric spherical if and only if 
each special representation appearing in the parabolic induction \eqref{eq:matringe} is a twist of $\St_2$ and any twist by a non-quadratic character 
appears accompanied by the inverse twist. 

Moreover $\dim\pi^J=2^{n-\rk N_{\pi}}$. In particular, $\dim\pi^J=1$ if and only if $\pi$ is maximally Steinberg.
\end{proposition}

\section{Improved  Ash--Ginzburg functionals}\label{s:intertwining}

Matringe's classification of generic representations of $G=\GL_{2n}(F)$ admitting Shalika models relies on an inductive
construction of Shalika models which is hard to control for the computation of local zeta functions. In this section, we
will explain how  using intertwining operators one can rescale the Ash--Ginzburg integral yielding a Shalika functional on any parahoric spherical representations.

\subsection*{Haar measures}

We fix a Haar measure $dx$ on $F$ giving $\mathcal{O}$ measure
1, which induces a measure $dX=\prod dx_{ij}$ on $M_n(F)$ giving
$M_n(\mathcal{O})$ measure 1. The Haar measure $dX$ satisfies $d(gX)=|\det g|^n dX$ and therefore yields the unimodular
Haar measure $dX/|\det X|^n$ on $\GL_n(F)$. We denote by
$dg$ a rescaling of this measure giving $K_n=\GL_n(\mathcal{O})$ volume 1. The
rescaling factor can be computed as the inverse of
\[\int_{K_n} dX/|\det X|^n=\int_{K_n} dX = \sum_{x\in \GL_n(\mathbb{F}_q)}\int_{x+\varpi M_n(\mathcal{O})}dX = q^{-n^2}\left|\GL_n(\mathbb{F}_q)\right|=\prod_{i=1}^{n}(1-q^{-i}).\]

The left Haar measure $db$ on $B_n$ is not unimodular and $d(hbh^{-1}) = \delta_B(h)db$, but its restriction to $B_n(\mathcal{O})$ is unimodular. We normalize $db$ to give $B_n\cap K_n$ measure 1. In this case we obtain quotient measures $dg$ on $B_n\backslash G_n$ and $B_n\cap K_n\backslash K_n$ which agree under the Cartan identification $B_n\backslash G_n=B_n\cap K_n\backslash K_n$.

\subsection{Statement of the local results} \label{sec:local}
Given a character $\theta=(\theta_1,\dots,\theta_{2n}) : T=(F^\times)^{2n}\to \mathbb{C}^\times$, the (unique) generic constituent $\pi$   
of  the normalized parabolic induction $I(\theta)=\Ind_B^G (\theta)$ admits a Shalika model if, and only if, (up to a permutation of the indices)
$\theta_i\theta_{n+i}=\mathbf{1}$ for all $1\leqslant i\leqslant n$. Recall that by  a result of Casselman, $\pi$ is Iwahori spherical (i.e.,  admits non-zero fixed vectors by the Iwahori subgroup $I$ of $G$) if, and only if,  $\theta$ is unramified. 
Henceforth we assume $\theta$ to be unramified, regular (i.e., $\theta_i\neq\theta_j$ for all  $1\leqslant i<j \leqslant 2n$)  and 
$\theta_i\theta_{n+i}=\mathbf{1}$ for all  $1\leqslant i<j \leqslant n$, later referred to as the Ash--Ginzburg ordering. 
 These conditions on $\theta$ force  (the unique) subrepresentation $\pi$ of $I(\theta)$ to be parahoric spherical. 
Given such a character $\theta$, using analytic continuation,  Ash--Ginzburg \cite{AG}  constructed a non-zero functional 
$\mathcal{S}^\theta_{\AG}:I(\theta)\to \mathbb{C}$ given by
\[\mathcal{S}^\theta_{\AG}(f)=\int_{G/B}\int_{M_n(F)}f
\left(\begin{pmatrix}
  & \mathbf{1}\\  \mathbf{1} & X\end{pmatrix}\begin{pmatrix}
h&\\&h\end{pmatrix}\right)\psi(-\tr X)dXdh.\]
 The  Ash--Ginzburg functional has been instrumental in computing local zeta integrals
for $\pi$ spherical (see \cite{AG,DJR,BDW,BDGJW}). If $I(\theta)$ is reducible with our generic $\pi$ occurring as 
 subrepresentation, then $\mathcal{S}^\theta_{\AG}|_\pi$ would still be  a Shalika functional  for $\pi$, provided it is non-zero. 
 We determine that this is the case precisely when $\pi$ has monodromy rank  $1$, allowing one to define $\mathcal{S}_\pi$
 simply as $\mathcal{S}^\theta_{\AG}|_\pi$. 
 
In the general case, we develop a theory of Shalika local coefficients and  define $\mathcal{S}_\pi$ by analytic continuation of Ash--Ginzburg integrals rescaled by Shalika local coefficients. Note that, when the monodromy rank of $\pi$ is $\geqslant 2$, then our methods imply the  existence of some non-generic, non-tempered 
parahoric spherical representations which would be a theme interesting to pursue on its own.


Given a regular unramified character $\theta$, for  $\tau\in W$ we let  $M_\tau:I(\theta^{\tau^{-1}})\to
I(\theta)$ be the intertwining operator.

 As  $B \backslash G/ J\cong  W/W_H$, a basis of $I(\theta)^J$ is given by   $\{ f^\theta_{BwJ}, [w]\in W/W_H\}$,  
where $f^\theta_{BwJ}:G\to \mathbb{C}$ is the unique function  supported on $BwJ$, normalized so that $f^\theta_{BwJ}(\begin{spmatrix} 1&\\&\varpi^{-\delta}\end{spmatrix}w)=1$. Our  focus now is  to determine explicit vectors in $\pi^J$ in the above decomposition.

The proof of the following result will occupy the entire section and be completed in   \S\ref{s:shalika-sub}. 
\begin{theorem}\label{t:shalika-sub}
Let $\pi\subset\Ind_B^G\theta$ be a parahoric spherical representation,  with   $\theta$ unramified, regular 
and such that $\theta_i\theta_{n+i}=\mathbf{1}$ for all $1\leqslant i\leqslant n$. For each $U$-eigenvalue $\alpha$ on $\pi^J$, 
there exists a Weyl element $\tau$ such that $\pi=\Im M_\tau$, and a  Shalika
functional $\mathcal{S}_\pi:\pi\to \mathbb{C}$ such that $\phi_\pi^\tau=M_\tau (f_{Bw_0J}^{\theta^{\tau^{-1}}})\in \pi^J[U-\alpha]$
such that   $\mathcal{S}_\pi(\phi_\pi^\tau)(t^{-\delta})=1$.
\end{theorem}

As a consequence we show existence of explicit Friedberg--Jacquet  test vectors beyond the spherical case. 

\begin{proposition}\label{p:zeta-local-constants}
Suppose $\widetilde\pi=(\pi,\alpha)$ and $\phi_\pi^\tau$ are as in Theorem \ref{t:shalika-sub}. Letting 
$W_{\widetilde\pi}=\mathcal{S}_\pi(\phi_\pi^\tau)$, one has 
\begin{enumerate}
\item If $\chi$ is an unramified character then
\begin{align*}
\zeta(s,W_{\widetilde\pi},\chi)&=q^{\delta s n - \delta n^2/2}\left(1-\frac{1}{q}\right)^n\chi(\varpi)^{-\delta n}\prod_{i=1}^n L(\theta^{\tau^{-1}}_{n+i}\chi,s).
\end{align*}
\item If $\chi$ is any character and $\beta$ is the conductor of $\chi$ (if ramified, otherwise $\beta=1$) then
\begin{align*}
\zeta(s,u^{-1}t^\beta\cdot W_{\widetilde\pi}, \chi)&=e_p(\widetilde{\pi},\chi,s).
\end{align*}
\end{enumerate}
\end{proposition}
\begin{proof}
The latter formula follows from \cite[Prop. 9.3]{BDGJW}. The
former is similar and we include the computations for the convenience of the reader.
Writing $\theta'=\theta^{\tau^{-1}}$, $F_0=f_{Bw_0J}^{\theta'}$ and recalling that $W_{\widetilde\pi}=\mathcal{S}_{\AG}^{\theta'}(F_0)$, we have
	\begin{align*}
		\zeta(s,  W_{\widetilde\pi} , \chi)&=\int_{G_n}W_{\widetilde\pi}\left[\begin{spmatrix}
			h&\\&1\end{spmatrix}\right]\chi(\det h)|\det
		h|^{s-1/2}dh\\
		&=\int_{G_n}\int_{K_n}\int_{M_n}F_0\left[\begin{spmatrix}
			&1\\1&X\end{spmatrix}\begin{spmatrix}
			kh&\\&k\end{spmatrix}\right] \chi(\det h)|\det
		h|^{s-1/2}\psi\big(-\tr X\big)\ dXdkdh\\
		&=\int_{G_n}\int_{K_n}\int_{M_n}F_0\left[\begin{spmatrix}
			&1\\1&X\end{spmatrix}\begin{spmatrix}
			h&\\&1\end{spmatrix}\right] \chi(\det h)|\det
		h|^{s-1/2}\psi\big(-\tr X\big)\ dXdkdh,
	\end{align*}
using that $\det khk^{-1}=\det h$, $d (khk^{-1})=d h$, and the
$J$-invariance of $F_0$. Writing $h=bk$ with $dh=dk\cdot db$ and using the fact that $\chi$ is unramified we get
\begin{align*}
		\zeta(s,  W_{\widetilde\pi} , \chi)
  &=\int_{B_n}\int_{M_n}F_0\left[\begin{spmatrix}
			&1\\1&X\end{spmatrix}\begin{spmatrix}
			b&\\&1\end{spmatrix}\right]  \chi(\det b)|\det
		b|^{s-1/2}\psi\big(-\tr X\big)\ dXdb\\
  &=\int_{B_n}\int_{M_n}(\delta_B^{1/2} \theta')\left[\begin{spmatrix} 1&\\&b\end{spmatrix}\right]F_0\left[\begin{spmatrix}
			&1\\1&b^{-1} X\end{spmatrix}\right]  \chi(\det b)|\det
		b|^{s-1/2}\psi\big(-\tr X\big)\ dXdb,
\end{align*}
Changing variables so that $b^{-1}X$ becomes $X$, noting that $d(bX)=|\det b|^n dX$, this equals
\begin{align*}
		\zeta(s,  W_{\widetilde\pi} , \chi)
  &=\int_{B_n}\int_{M_n}\delta_B^{1/2} \theta'\left[\begin{spmatrix} 1&\\&b\end{spmatrix}\right]F_0\left[\begin{spmatrix}
			&1\\1&X\end{spmatrix}\right] \chi(\det b)|\det
		b|^{s+n-1/2}\psi\big(-\tr bX\big)\ dXdb.
\end{align*}
As in \cite[Prop. 9.3]{BDGJW}, $X\in M_n(\mathcal{O})$ and thus $b\in B_n(F)\cap \varpi^{-\delta}M_n(\mathcal{O})$. Since $F_0(\begin{spmatrix} 1&\\&\varpi^{-\delta}\end{spmatrix}w_0)=1$, we have $F_0(w_0)=\delta_B^{\frac{1}{2}}\theta'(\begin{spmatrix} 1&\\&\varpi^{\delta}\end{spmatrix})$, so after integrating out $X$ we get
\begin{align*}
		\zeta(s,  W_{\widetilde\pi} , \chi)
  &=\int_{B_n\cap \varpi^{-\delta}M_n(\mathcal{O})}\delta_B^{1/2} \theta'\left[\begin{spmatrix} 1&\\&\varpi^\delta b\end{spmatrix}\right]\chi(\det b)|\det
  b|^{s+n-1/2}\ db\\
  &=\delta_B^{\frac{1}{2}}\theta'(\begin{spmatrix} 1&\\&\varpi^{\delta}\end{spmatrix})\prod_{i=1}^n\int_{\varpi^{-\delta}\mathcal{O}}\theta'_{n+i}\chi|\cdot|^{s-1}(t_i)dt_i
  =q^{\delta s n - \delta n^2/2}\left(1-\tfrac{1}{q}\right)^n\chi(\varpi)^{-\delta n}L(\theta'_{n+i}\chi,s).
\end{align*}
  Here we used that if we write 
  $b=\begin{pmatrix} t_1 & & u_{ij}\\
		&\ddots&\\
		0&& t_n\end{pmatrix}$, 
then $db = \prod|t_i|^{i-n-1}\prod dt_i\prod du_{ij}$.
\end{proof} 

\subsection{Shalika local coefficients}

Suppose $\theta$  an unramified regular  character with the Ash--Ginzburg ordering and $\tau\in W$ such that $\theta^{\tau^{-1}}$ is also
in the Ash--Ginzburg ordering. If $ I(\theta)$ is a
irreducible then the intertwining operator
$M_\tau: I(\theta^{\tau^{-1}})\to I(\theta)$ is an isomorphism, yielding a diagram
\[\xymatrix{I(\theta^{\tau^{-1}}) \ar[rr]^{M_\tau}\ar[rd]_{\mathcal{S}_{\AG}^{\theta^{\tau^{-1}}}}&& I(\theta)\ar[ld]^{\mathcal{S}^{\theta}_{\AG}}\\
  &\Ind_S^G\psi&
}\]
By multiplicity one for  Shalika models, there exists a constant $\mathcal{C}_{\theta,\tau}$ such that
$\mathcal{S}_{\theta}\circ M_\tau = \mathcal{C}_{\theta,\tau}\mathcal{S}_{\theta^{\tau^{-1}}}$.

Whenever $\ell(\tau\tau')=\ell(\tau)+\ell(\tau')$, $M_{\tau'\tau}=M_\tau\circ M_{\tau'}$
and therefore $\mathcal{C}_{\theta,
  \tau'\tau}=\mathcal{C}_{\theta^{\tau^{-1}},\tau'}\mathcal{C}_{\theta,\tau}$. We
conclude that, to determine $\mathcal{C}_{\theta,\tau}$, it
suffices to compute it for $\tau$ of the following forms:
\begin{enumerate}
\item $\tau=(i,i+1)(n+i,n+i+1)$ for some $i<n$, and
\item $\tau=(n,2n)=\nu^{-1} (n,n+1)\nu$, where $\nu= (n+1,n+2)\cdots (2n-1,2n)$.
\end{enumerate}

\subsection{The image of the open cell under intertwining operators}
For $\theta$ unramified regular character   of $T$, we let 
\[c_{i,j}(\theta)=\dfrac{1-q^{-1}\theta_i\theta_j^{-1}(\varpi)}{1-\theta_i\theta_j^{-1}(\varpi)}\quad \text{ and }\quad d_{i,j}(\theta)=1-c_{j,i}(\theta)=c_{i,j}(\theta)-q^{-1}=\dfrac{1-q^{-1}}{1-\theta_i\theta_j^{-1}(\varpi)}.\] 

We begin by computing the image of $f_{BwJ}^\theta$ under a simple
intertwining operator.
\begin{lemma}\label{l:parahoric-intertwining}
Suppose $\alpha=(i,i+1)\in W$ is simple, and $w\in W/W_H$. If
$w^{-1}\alpha w\in W_H$ then
\[M_\alpha f^\theta_{BwJ}=c_{i,i+1}(\theta)f^{\alpha \theta}_{BwJ}.\]
Otherwise,
\[M_\alpha f^\theta_{BwJ}=\begin{cases}
-d_{i+1,i}(\theta)f_{BwJ}^{\alpha
  \theta}+q^{-1}f^{\alpha \theta}_{B \alpha wJ}&\text{ if }\ell(\alpha w)>\ell(w),\\
d_{i,i+1}(\theta)f_{BwJ}^{\alpha
  \theta}+f^{\alpha \theta}_{B \alpha wJ}&\text{ if }\ell(\alpha w)<\ell(w). 
\end{cases}
\]
\end{lemma}
\begin{proof}
By
\cite[Thm. 3.4]{casselman:unramified}, if $\alpha$ is simple and
$w\in W$ then
\[M_\alpha
f^\theta_{BwI}=\begin{cases}(c_{i,i+1}(\theta)-1)f_{BwI}^{\alpha
  \theta}+q^{-1}f^{\alpha \theta}_{B \alpha wI}&\text{ if }\ell(\alpha w)>\ell(w),\\
(c_{i,i+1}(\theta)-q^{-1})f_{BwI}^{\alpha
  \theta}+f^{\alpha \theta}_{B \alpha wI}&\text{ if }\ell(\alpha w)<\ell(w). 
\end{cases}\]
Therefore, the parahoric function $M_\alpha f_{BwJ}^\theta$ has support in $BwJ\cup B \alpha w J$. In the latter case, the two cells are disjoint, while in the former they coincide. To compute the coefficients, it is enough to plug in $w$ into the above formula.
\end{proof}

\begin{proposition}\label{p:complicated-intertwining}
Let  $\theta$ be an unramified regular character   of $T$. For all $1\leqslant i \leqslant n-1$, we have 
\begin{align*}
M_{(i,n)(n+i,2n)}f_{Bw_0J}^\theta&=c_{i,n}(\theta)c_{n+i,2n}(\theta)\prod_{i<j<n}c_{i,j}(\theta)c_{n+i,n+j}(\theta)c_{j,n}(\theta)c_{n+j,2n}(\theta)f_{Bw_0J}^{\theta^{(i,n)(n+i,2n)}}. 
\end{align*}
Moreover, letting  $\theta'=\theta^{(n,2n)}=(\theta_1,\cdots,\theta_{n-1},\theta_{2n},\theta_{n+1},\cdots, \theta_{2n-1},\theta_n)$, 
we have 
\begin{align*}
  M_{\nu^{-1}(n,n+1)\nu}f_{Bw_0J}^\theta&=\prod_{i=1}^{n-1}c_{n+i,2n}(\theta)\left(\prod_{i=1}^{n-1}c_{(n,n+i)}(\theta)d_{n,2n}(\theta)F_0^{\theta'}+\sum_{k=1}^{n-1}\prod_{i=k+1}^{n-1}c_{(n,n+i)}(\theta)d_{n,n+k}(\theta)
    F_k^{\theta'} + F_n^{\theta'}\right)\\
  &=\prod_{i=1}^{n-1}c_{n+i,2n}(\theta)\left(F_n^{\theta'}+d_{n,2n-1}(\theta)F_{n-1}^{\theta'}+d_{n,2n-2}(\theta)c_{n,2n-1}(\theta)F_{n-2}^{\theta'}+\cdots\right),
\end{align*}
where for $1\leqslant k\leqslant n$ we have denoted  $F_k^\theta=f^\theta_{B  (n+k,\ldots,n)w_0J}\in I(\theta)^J$.
\end{proposition}
\begin{proof}
The expression above follows inductively using the following three
observations, the last two of which are direct consequences of Lemma \ref{l:parahoric-intertwining}:
\begin{itemize}
\item Every time we apply a simple transposition, the length decreases, as $w_0$ is the longest Weyl element,
\item $M_\nu f^\theta_{Bw_0J} =
\prod_{i=1}^{n-1}c_{n+i,2n}(\theta) f_{Bw_0J}^{\nu\theta}$, and
\item for $0\leqslant k\leqslant n$ we have $M_{(n+i,n+i+1)}F_k^{\theta} = c_{(n+i,n+i+1)}(\theta)F_k^{(n+i,n+i+1)\theta}$ if $i>k$ and 
\[M_{(n+k,n+k+1)}F_k^{\theta} = (c_{(n+k,n+k+1)}(\theta)-q^{-1})F_k^{(n+k,n+k+1)\theta}+F_{k+1}^{(n+k,n+k+1)\theta}.\] 
\end{itemize}
Indeed,  $(n,\ldots, n+k)(n+i,n+i+1)(n+k,\ldots,n)=(n+i,n+i+1)\in W_H$ for  $i>k$,  whereas 
$(n,\ldots, n+k)(n+k,n+k+1)(n_k,\ldots,n)=(n,n+k+1)\notin W_H$.
\end{proof}

\subsection{Shalika functions of certain test functions}\label{ss:shalika-at1}
To determine Shalika local coefficients, it suffices to compute the value of Shalika functions as 
$t^{-\delta}\varpi^\delta =\begin{spmatrix} 1&\\&\varpi^\delta\end{spmatrix}$. Writing $\mathcal{D}=(\theta\delta_B^{1/2})^{w_0}(t^{-\delta}\varpi^\delta )=q^{-\delta n^2/2}\prod\limits_{i=1}^n \theta_i(\varpi)^\delta$
\begin{align*}
\mathcal{S}_{\theta}(F^\theta_k)(t^{-\delta}\varpi^\delta )&=\int_{M_n(F)}F^\theta_k
                                     \left(\begin{pmatrix}&1\\1&X\end{pmatrix}\begin{pmatrix}
                                     1&\\&\varpi^{\delta}\end{pmatrix}\right)\psi(-\tr
  X)dX\\
  &=\int_{M_n(F)}F^\theta_k \left(\begin{pmatrix} &1\\1&\end{pmatrix}t^{-\delta}\varpi^\delta \begin{pmatrix} 1&\varpi^{\delta}X\\&1\end{pmatrix}\right)\psi(-\tr X)dX\\
  &=q^{\delta n}\mathcal{D} \int_{M_n(F)}F^\theta_k \begin{pmatrix}&1\\1&X\end{pmatrix}\psi(-\tr(\varpi^{-\delta}X))dX
\end{align*}
As  $F_k^\theta$ is $J$-invariant, for $n\in N_n(\mathcal{O})$, we have $F_k^\theta \begin{pmatrix}&1\\1&X\end{pmatrix}=F_k^\theta \left(\begin{pmatrix} &1\\1&\end{pmatrix}\begin{pmatrix} n^{-1}&\\&1\end{pmatrix}\begin{pmatrix} 1&X\\&1\end{pmatrix}\begin{pmatrix} n&\\&1\end{pmatrix}\right)$ yielding 
\begin{align*}
\mathcal{S}(F_k^\theta)(t^{-\delta}\varpi^\delta )&=q^{\delta n}\mathcal{D}\int_{M_n(F)}\int_{N_n(\mathcal{O})}F_k^\theta \left(\begin{pmatrix}
&1\\1&\end{pmatrix}\begin{pmatrix}
n^{-1}&\\&1\end{pmatrix}\begin{pmatrix}
1&X\\&1\end{pmatrix}\begin{pmatrix}
n&\\&1\end{pmatrix}\right)\psi(-\tr(\varpi^{-\delta}X))dXdn\\
 &= q^{\delta n}\mathcal{D}\int_{M_n(F)}\int_{N_n(\mathcal{O})}F_k^\theta \begin{pmatrix}&1\\1&n^{-1} X\end{pmatrix}\psi(-\tr(\varpi^{-\delta}X))dXdn\\
  &=q^{\delta n}\mathcal{D}\int_{M_n(F)} F_k^\theta \begin{spmatrix}&1\\1&X\end{spmatrix} \left(\int_{N_n(\mathcal{O})}\psi(-\tr(\varpi^{-\delta} nX))dn\right)dX
\end{align*}
Observe that $\displaystyle \int_{N_n(\mathcal{O})}\psi(-\tr(\varpi^{-\delta}nX))dn$ vanishes unless $X=(x_{ij})$ with $x_{ij}\in \mathcal{O}$ for
all $i>j$. The $J$-invariance of $F_k^\theta$ allows to integrate these variables out, hence we may assume that  $X$ is upper triangular. 
\begin{align*}
\mathcal{S}(F^\theta_k)(t^{-\delta}\varpi^\delta )&=q^{\delta n}\mathcal{D}\int_{B_n(F)} F_k^\theta
                  \begin{pmatrix}&1\\1&X\end{pmatrix}\psi(-\tr(\varpi^{-\delta}X))dX.
\end{align*}

To compute the above integral, it will suffice to restrict out attention to those upper triangular $X$ such that 
$\begin{pmatrix}&1\\1&X\end{pmatrix}\in \supp(F_k^\theta)=B(n+k,\ldots,n)w_0J$. Suppose      $\begin{pmatrix}a&b\\&c\end{pmatrix}\in B$ such that
\[\begin{pmatrix}a&b\\&c\end{pmatrix}\begin{pmatrix} &1\\1&X\end{pmatrix}=\begin{pmatrix} b&a+bX\\c&cX\end{pmatrix}\in (n+k,\ldots,n)w_0J=\begin{pmatrix} \varpi M_n(\mathcal{O})&K_n\\K_n&M_n(\mathcal{O})\end{pmatrix}.\]
We deduce that
\[\left(
\begin{array}{ccccccc|ccccccc}
&&&b_{ij},i<n&&&&&&&(a+bX)_{ij},i<n&&&\\
  \hdashline
&&&c_k&\ldots&&&&&&c_kx_k&\ldots&&  \\
  \hline
  b_{n1}&b_{n2}&\ldots &&&&b_{nn} & &&&(a+bX)_{n,*}&&&\\
\hdashline
  c_{1}&c_{12}&\ldots &&&&c_{1n}&c_{1}x_{1}&\ldots&&&&&\\
        &\ddots&       &&&&      &           &\ddots&&&&&\\
        &      &c_{k-1}&\ldots &&&c_{k-1,n}& &&
                                                c_{k-1}x_{k-1}&\ldots&&&\\
  \hdashline
       &&&&c_{k+1}&\ldots &c_{k+1,n}&&&&&c_{k+1}x_{k+1}&\ldots&\\
       &&&&&\ddots &&&&&&&\ddots&\\
  &&&&&&c_n&&&&&&&c_nx_n
\end{array}
\right)\]
is in $\begin{pmatrix} \varpi
M_n(\mathcal{O})&K_n\\K_n&M_n(\mathcal{O})\end{pmatrix}$. First, note that
\[\begin{pmatrix}
c_{k+1}&\ldots&\\&\ddots&\\&&c_n\end{pmatrix}\in K_{n-k-1},\ \ \ \begin{pmatrix}
c_{k+1}x_{k+1}&\ldots&\\&\ddots&\\&&c_nx_n\end{pmatrix}\in
M_{n-k-1}(\mathcal{O})\]
which implies that $x_{ij}\in \mathcal{O}$ for $i>k$. As before,
we can integrate these variables out and assume that $x_{ij}=0$
for $i>k$. Similarly, we note that
\[\begin{pmatrix}          b_{n1}&b_{n2}&\ldots &&&&b_{nn}\\
c_{1}&c_{12}&\ldots &&&&c_{1n}\\
        &\ddots&       &&&& \\
        &      &c_{k-1}&\ldots &&&c_{k-1,n} \end{pmatrix}\in K_k,\ \ \ \ \begin{pmatrix}         &&&(a+bX)_{n,*}&&&\\
        c_{1}x_{1}&\ldots&&&&&\\
        &\ddots&&&&&\\
        &&
        c_{k-1}x_{k-1}&\ldots&&&\\
\end{pmatrix}\in M_{k-1,n}(\mathcal{O}).\]
Dropping the last column of the matrix on the right means that
we can ignore the upper triangular matrix $a$. We see that
\[\begin{pmatrix}         &&&(bX)_{n,<n}&&&\\
c_{1}x_{1}&\ldots&&&&&\\
        &\ddots&&&&&\\
        &&
        c_{k-1}x_{k-1}&\ldots&&&
        \end{pmatrix}=\begin{pmatrix}  b_{n1}&b_{n2}&\ldots &b_{nk}\\
        c_{1}&c_{12}&\ldots &c_{1k}\\
        &\ddots&       & \\
        &      &c_{k-1}&c_{k-1,k}
         \end{pmatrix}
        \begin{pmatrix}x_1&\ldots&&&x_{1,n-1}\\
        &\ddots&&&\\
        &&x_{k-1}&\ldots&x_{k-1,n-1}\end{pmatrix}\]
        has integral entries.

Again, we deduce that $x_{ij}\in \mathcal{O}$ (and therefore can
be assumed to be 0) for $i<k,j<n$.

Finally, from the parahoric matrix we deduce that the row
$c_k,\ldots, c_{k,n}$ lies in $(\varpi)$, and the row $0,\ldots,
0, c_kx_{k,n}$ is a row in $K_n$. Therefore, $c_k\in (\varpi)$
and $c_k x_{k,n}\in \mathcal{O}^\times$, which implies that
$x_{k,n}\in F-\mathcal{O}$.

Putting everything together, we see that
\[\mathcal{S}(F^\theta_k)(t^{-\delta}\varpi^\delta )=q^{\delta n}\mathcal{D}\int_{X} F_k^\theta\left(\begin{pmatrix}&1\\1&X\end{pmatrix}\right)\psi(-\tr(\varpi^{-\delta}X))dX,\]
where the integral is taken over matrices $X$ all of whose
entries are 0 except the top $k$ on the $n$-th column, the lowest entry of which is in $F-\mathcal{O}$.

What about this last column? We'll show that if $X$ is of this form then $\begin{pmatrix}&1\\1&X\end{pmatrix}\in B(n+k,\ldots,n)w_0J$ whenever $v(z_k)<0$. In fact, we'll show that in this case we can choose our matrix $\begin{pmatrix} a&b\\&c\end{pmatrix}$ such that  
\[
\begin{pmatrix}
b_{n,1}&\ldots&&b_{n,k}\\
c_1&\ldots &&c_{1,k}\\
& \ddots &&\\
&&c_{k-1}&c_{k-1,k}
\end{pmatrix}\begin{pmatrix} z_1\\z_2\\\vdots\\z_k\end{pmatrix}=\begin{pmatrix} -a_n\\0\\\vdots\\0\end{pmatrix},
\]
in which case, up to a unit, $z_k = c_1\cdots c_{k-1}a_n$ and
$c_k = z_k^{-1}$. 

Suppose $v(z_k)<v(z_{k-1})$. Then we can choose $c_{k-1}=1$,
$c_{k-1,k}=-\frac{z_{k-1}}{z_k}$, in which case the above matrix
identity becomes
\[\begin{pmatrix}
b_{n,1}&\ldots&b_{n,k-2}&b'_{n,k}\\
c_1&\ldots &c_{1,k-2}&c'_{1,k}\\
& \ddots &&\\
&&c_{k-2}&c'_{k-2,k}
\end{pmatrix}\begin{pmatrix}
z_1\\\vdots\\z_{k-2}\\z_k\end{pmatrix}=\begin{pmatrix}
-a_n\\0\\\vdots\\0\end{pmatrix},\]
where $b'_{n,k}=b_{n,k}+b_{n,k-1}\frac{z_{k-1}}{z_k}$ and
$c'_{i,k}=c_{i,k}+c_{i,k-1}\frac{z_{k-1}}{z_k}$. The left matrix
is still in $K_{k-1}$ and (up to a unit) $z_k=ac_1\cdots c_{k-2}$.

If $v(z_k)\geq v(z_{k-1})$, we can choose $c_{k-1}=-\frac{z_k}{z_{k-1}}$ and $c_{k-1,k}=1$. Again, we get 
\[\begin{pmatrix}
b_{n,1}&\ldots&b_{n,k-2}&b'_{n,k-1}\\
c_1&\ldots &c_{1,k-2}&c'_{1,k-1}\\
& \ddots &&\\
&&c_{k-2}&c'_{k-2,k-1}
\end{pmatrix}\begin{pmatrix}
z_1\\\vdots\\z_{k-2}\\z_{k-1}\end{pmatrix}=\begin{pmatrix}
-a_n\\0\\\vdots\\0\end{pmatrix},\]
where $b'_{n,k-1}=b_{n,k-1}+b_{n,k}\frac{z_{k}}{z_{k-1}}$ and
$c'_{i,k-1}=c_{i,k-1}+c_{i,k}\frac{z_{k}}{z_{k-1}}$. The left matrix
is still in $K_{k-1}$ and (up to a unit) $z_{k-1}=ac_1\cdots
c_{k-2}$.

\subsubsection{The computation for $F_k^\theta$ when $k<n$}\label{ss:shalikaFk}
In this case, $\tr X=0$ and we are computing
\begin{equation}\label{eq:Fk}\mathcal{S}_\theta(F^\theta_k)(t^{-\delta}\varpi^\delta )=q^{\delta n}\mathcal{D}\int_{v(z_k)<0} F^\theta_k \begin{spmatrix}&1\\1&X\end{spmatrix}dz_1\ldots dz_k.
\end{equation}
For unramified characters $\eta,\eta_1,\ldots, \eta_{k-1}$ let
\[\mathcal{I}(\eta,\eta_1,\ldots,\eta_{k-1})=\int_{v(z_k)<0}\eta|\cdot|^{-k}(a)\prod_{i=1}^{k-1}\eta_i|\cdot|^{-(k-i)}(c_i)dz_1\cdots dz_k,\]
in which case \eqref{eq:Fk} becomes
\begin{align*}&=q^{\delta n}\int_{v(z_k)<0}\theta_n|\cdot|^{\frac{1}{2}}(a^{-1})\theta_{n+1}|\cdot|^{\frac{1}{2}-1}(c_1^{-1})\cdots
    \theta_{n+k-1}|\cdot|^{\frac{1}{2}-(k-1)}(c_{k-1}^{-1})\theta_{n+k}|\cdot|^{\frac{1}{2}-k}(ac_1\cdots
    c_k)dz_1\cdots dz_k\\
  &=q^{\delta n}\int_{v(z_k)<0}\theta_n^{-1}\theta_{n+k}|\cdot|^{-k}(a)\prod_{i=1}^{k-1}\theta_{n+i}^{-1}\theta_{n+k}|\cdot|^{i-k}(c_i)dz_1\cdots
    dz_k\\
  &=q^{\delta n}\mathcal{I}(\frac{\theta_{n+k}}{\theta_n},\frac{\theta_{n+k}}{\theta_{n+1}},\ldots,\frac{\theta_{n+k}}{\theta_{n+k-1}}).
\end{align*}
\begin{lemma}\label{l:I}
We have $\displaystyle \mathcal{I}(\eta) =
\left(1-\frac{1}{q}\right)\frac{\eta(\varpi)^{-1}}{1-\eta(\varpi)^{-1}}$ and
\[\mathcal{I}(\eta,\eta_1,\ldots,\eta_{k-1})=c(\eta_{k-1})\mathcal{I}(\eta,\eta_1,\ldots,\eta_{k-2})-d(\eta_{k-1})\mathcal{I}(\eta\eta_{k-1}^{-1},\eta_1\eta_{k-1}^{-1},\ldots,\eta_{k-2}\eta_{k-1}^{-1}).\]
\end{lemma}
\begin{proof}
The first part is a straightforward computation. For the second part, we compute
\begin{align*}
&\mathcal{I}(\eta,\eta_1,\ldots,\eta_{k-1})=\int_{v(z_k)<0,v(z_k)<v(z_{k-1})}\eta|\cdot|^{-k}(a)\prod_{i=1}^{k-1}\eta_i|\cdot|^{i-k}(c_i)dz_1\cdots
                                            dz_k\\
  &\ +\int_{v(z_{k-1})\leqslant
                           v(z_k)<0}\eta|\cdot|^{-k}(a)\prod_{i=1}^{k-1}\eta_i|\cdot|^{i-k}(c_i)dz_1\cdots
                                            dz_k\\
  &=\int_{v(z_k)<0,v(z_k)<v(z_{k-1})}\eta|\cdot|^{-k}(a)\prod_{i=1}^{k-2}\eta_i|\cdot|^{i-k}(c_i)dz_1\cdots
    dz_k\\
  &\ +\int_{v(z_{k-1})\leqslant v(z_k)<0}\eta|\cdot|^{-k}(a)\prod_{i=1}^{k-2}\eta_i|\cdot|^{i-k}(c_i) \eta_{k-1}|\cdot|^{-1}(-\frac{z_k}{z_{k-1}})dz_1\cdots
                                            dz_k\\
  &=\int_{v(z_k)<0}\eta|\cdot|^{-k}(a)\prod_{i=1}^{k-2}\eta_i|\cdot|^{i-k}(c_i)dz_1\cdots
    dz_{k-2}dz_k\int_{v(z_k)<v(z_{k-1})}dz_{k-1}\\
  &\ +\int_{v(z_{k-1})<0}\eta|\cdot|^{-(k-1)}(a)\prod_{i=1}^{k-2}\eta_i|\cdot|^{i-k-1}(c_i)\eta_{k-1}(z_{k-1})^{-1} dz_1\cdots dz_{k-1} \int_{v(z_{k-1})\leqslant v(z_k)<0}\eta_{k-1}|\cdot|^{-1}(z_k) dz_k\\
  &=\int_{v(z_k)<0}\eta|\cdot|^{-k}(a)\prod_{i=1}^{k-2}\eta_i|\cdot|^{i-k}(c_i)|z_{k-1}\varpi|dz_1\cdots
    dz_{k-2}dz_k\\
  &\ +\frac{1-q^{-1}}{1-\eta_{k-1}(\varpi)}\int_{v(z_{k-1})<0}\eta|\cdot|^{-(k-1)}(a)\prod_{i=1}^{k-2}\eta_i|\cdot|^{i-k-1}(c_i) (1-\eta_{k-1}(z_{k-1})^{-1})dz_1\cdots dz_{k-1}\\
  &=q^{-1}\mathcal{I}(\eta,\eta_1,\ldots,\eta_{k-2})
    +\frac{1-q^{-1}}{1-\eta_{k-1}(\varpi)}\left(\mathcal{I}(\eta,\eta_1,\ldots,\eta_{k-2})-\mathcal{I}(\eta\eta_{k-1}^{-1},\eta_1\eta_{k-1}^{-1},\ldots,\eta_{k-2}\eta_{k-1}^{-1})\right)\\
  &=\frac{1-q^{-1}\eta_{k-1}(\varpi)}{1-\eta_{k-1}(\varpi)}\mathcal{I}(\eta,\eta_1,\ldots,\eta_{k-2})-\frac{1-q^{-1}}{1-\eta_{k-1}(\varpi)}\mathcal{I}(\eta\eta_{k-1}^{-1},\eta_1\eta_{k-1}^{-1},\ldots,\eta_{k-2}\eta_{k-1}^{-1}).\qedhere 
\end{align*}
\end{proof}

We conclude the following recursive formula:
\begin{align*}
q^{-\delta n}\mathcal{S}_\theta(F^\theta_k)(t^{-\delta}\varpi^\delta )&=\mathcal{I}(\frac{\theta_{n+k}}{\theta_n},\frac{\theta_{n+k}}{\theta_{n+1}},\ldots,\frac{\theta_{n+k}}{\theta_{n+k-1}})\\
                                                                                       &=c_{(n+k,n+k-1)}(\theta)\mathcal{I}(\frac{\theta_{n+k}}{\theta_n},\frac{\theta_{n+k}}{\theta_{n+1}},\ldots,\frac{\theta_{n+k}}{\theta_{n+k-2}}) -d_{(n+k,n+k-1)}(\theta)\mathcal{I}(\frac{\theta_{n+k-1}}{\theta_n},\frac{\theta_{n+k-1}}{\theta_{n+1}},\ldots,\frac{\theta_{n+k-1}}{\theta_{n+k-2}}).
\end{align*}

\subsubsection{The computation for $F_n^\theta$}\label{ss:shalikaFn}
In this case, $\tr X=z_n$ and we are computing
\begin{equation}\label{eq:Fn}\mathcal{S}_\theta(F^\theta_n)(t^{-\delta}\varpi^\delta )=q^{\delta n}\mathcal{D}\int_{v(z_n)<0} F^\theta_n \left(\begin{pmatrix}&1\\1&X\end{pmatrix}\right)dz_1\ldots dz_n.
\end{equation}
For unramified characters $\eta,\eta_1,\ldots, \eta_{k-1}$ let
\[\mathcal{I}_\psi(\eta,\eta_1,\ldots,\eta_{k-1})=\int_{v(z_k)<0}\eta|\cdot|^{-k}(a)\prod_{i=1}^{k-1}\eta_i|\cdot|^{i-k}(c_i)\psi(-z_k)dz_1\cdots
dz_k.\]
In this case, we get
\begin{align*}
\mathcal{S}_\theta(F^\theta_n)(t^{-\delta}\varpi^\delta )
  &=q^{\delta n}\int_{v(z_n)<0}\theta_n^{-1}\theta_{2n}|\cdot|^{-n}(a)\prod_{i=1}^{n-1}\theta_{n+i}^{-1}\theta_{2n}|\cdot|^{i-k}(c_i)\psi(-z_n)dz_1\cdots
    dz_k\\
  &=q^{\delta n}\mathcal{I}_\psi\left(\frac{\theta_{2n}}{\theta_n},\frac{\theta_{2n}}{\theta_{n+1}},\ldots,\frac{\theta_{2n}}{\theta_{2n-1}}\right).
\end{align*}

As in the case $k<n$, we will compute this integral recursively:
\begin{lemma}\label{l:Ipsi}
We have
$\mathcal{I}_\psi(\eta)=-q^{-1}\eta(\varpi)^{-1}$ and
\[\mathcal{I}_\psi(\eta,\eta_1,\ldots,\eta_{k-1})=q^{-1}\mathcal{I}_\psi(\eta,\eta_1,\ldots,\eta_{k-2})-\frac{1}{q\eta_{k-1}(\varpi)}\mathcal{I}(\eta\eta_{k-1}^{-1},\eta_1\eta_{k-1}^{-1},\ldots,\eta_{k-2}\eta_{k-1}^{-1}).\]
\end{lemma}
\begin{proof}
The first part follows from the fact that (with $\delta=0$)
\[\int_{v(z)=m}\eta(z)\psi(-z)dz=\begin{cases}
-\eta(\varpi)^{-1}&m=-1,\\ 0&m<-1.
\end{cases}
\]
For the second part, we compute
\begin{align*}
&\mathcal{I}_\psi(\eta,\eta_1,\ldots,\eta_{k-1})=\int_{v(z_k)<0,v(z_k)<v(z_{k-1})}\eta|\cdot|^{-k}(a)\prod_{i=1}^{k-1}\eta_i|\cdot|^{i-k}(c_i)\psi(-z_k)dz_1\cdots
                                                 dz_k\\
  &\ +\int_{v(z_{k-1})\leqslant
                           v(z_k)<0}\eta|\cdot|^{-k}(a)\prod_{i=1}^{k-1}\eta_i|\cdot|^{i-k}(c_i)\psi(-z_k)dz_1\cdots
                                            dz_k\\
  &=\int_{v(z_k)<0,v(z_k)<v(z_{k-1})}\eta|\cdot|^{-k}(a)\prod_{i=1}^{k-2}\eta_i|\cdot|^{i-k}(c_i)\psi(-z_k)dz_1\cdots
    dz_k\\
  &\ +\int_{v(z_{k-1})\leqslant v(z_k)<0}\eta|\cdot|^{-k}(a)\prod_{i=1}^{k-2}\eta_i|\cdot|^{i-k}(c_i) \eta_{k-1}|\cdot|^{-1}(-\frac{z_k}{z_{k-1}})\psi(-z_k)dz_1\cdots
                                            dz_k\\
  &=\int_{v(z_k)<0}\eta|\cdot|^{-k}(a)\prod_{i=1}^{k-2}\eta_i|\cdot|^{i-k}(c_i)\psi(-z_k)dz_1\cdots
    dz_{k-2}dz_k\int_{v(z_k)<v(z_{k-1})}dz_{k-1}\\
  &\ +\int_{v(z_{k-1})<0}\eta\eta_{k-1}^{-1}|\cdot|^{-(k-1)}(a)\prod_{i=1}^{k-2}\eta_i\eta_{k-1}^{-1}|\cdot|^{i-k-1}(c_i) dz_1\cdots dz_{k-1} \int_{v(z_{k-1})\leqslant v(z_k)<0}\eta_{k-1}|\cdot|^{-1}(z_k)\psi(-z_k) dz_k\\
  &=\int_{v(z_k)<0}\eta|\cdot|^{-k}(a)\prod_{i=1}^{k-2}\eta_i|\cdot|^{i-k}(c_i)|z_{k-1}\varpi|\psi(-z_k)dz_1\cdots
    dz_{k-2}dz_k\\
  &\ -\frac{1}{q\eta_{k-1}(\varpi)}\int_{v(z_{k-1})<0}\eta(a)\prod_{i=1}^{k-2}\eta\eta_{k-1}^{-1}|\cdot|^{-(k-1)}(a)\prod_{i=1}^{k-2}\eta_i\eta_{k-1}^{-1}|\cdot|^{i-k-1}(c_i)dz_1\cdots dz_{k-1}\\
  &=q^{-1}\mathcal{I}_\psi(\eta,\eta_1,\ldots,\eta_{k-2})-\frac{1}{q\eta_{k-1}(\varpi)}\mathcal{I}(\eta\eta_{k-1}^{-1},\eta_1\eta_{k-1}^{-1},\ldots,\eta_{k-2}\eta_{k-1}^{-1}). \qedhere
\end{align*}
\end{proof}
Therefore
\begin{align*}
q^{-\delta n}\mathcal{S}_\theta(F^\theta_n)(t^{-\delta}\varpi^\delta )&=\mathcal{I}_\psi(\frac{\theta_{2n}}{\theta_n},\frac{\theta_{2n}}{\theta_{n+1}},\ldots,\frac{\theta_{2n}}{\theta_{2n-1}})\\
  &=q^{-1}\mathcal{I}_\psi(\frac{\theta_{2n}}{\theta_n},\frac{\theta_{2n}}{\theta_{n+1}},\ldots,\frac{\theta_{2n}}{\theta_{2n-2}})-\frac{\theta_{2n-1}}{\theta_{2n}}(\varpi)\mathcal{I}(\frac{\theta_{2n-1}}{\theta_n},\frac{\theta_{2n-1}}{q\theta_{n+1}},\ldots,\frac{\theta_{2n-1}}{\theta_{2n-2}})
\end{align*}

\subsection{Factorization}
The main result of this section is a factorization of the Shalika local coefficient. We begin with a technical result on factoring rational functions  subject to the recursion relations from \S\ref{ss:shalika-at1}.

Consider the rational functions $c(x) = \frac{1-q^{-1}x}{1-x}$,
$d(x)=\frac{1-q^{-1}}{1-x}$ and $A(x_1,\ldots, x_n)$ and
$B(x_1,\ldots,x_n)$ defined recursively by
\begin{align*} 
A(x_1,\ldots, x_n)&=c(x_n)A(x_1,\ldots,x_{n-1})-d(x_n)A(x_1x_n^{-1},\ldots,x_{n-1}x_n^{-1})&A(x_1)&=-d(x_1)\\
B(x_1,\ldots, x_n)&=\frac{1}{q}B(x_1,\ldots, x_{n-1})-\frac{1}{qx_n}B(x_1x_n^{-1},\ldots,x_{n-1}x_n^{-1})&B(x_1)&=-\frac{1}{qx_1}
\end{align*}
for $n\geq 2$, with the convention that $A(\emptyset)=1$.
\begin{lemma}\label{l:fact-AB}
For $n\geq 1$ we have
\begin{align} \label{eq:fact-A}
&B(x_1,\ldots,x_n)-d
\left(\frac{1}{qz}\right)^{-1}A(x_1z^{-1},\ldots,
x_nz^{-1})=-\frac{1}{x_1}c \left(\frac{1}{x_1}\right)c
\left(\frac{x_2}{x_1}\right)\cdots c
\left(\frac{x_n}{x_1}\right)\frac{1-x_1^{-1}}{z^{-1}-x_1^{-1}}, \\ \label{eq:fact-B}
&B(x_1,\ldots, x_n)+\sum_{k=1}^{n}\left(d(x_k)\prod_{j=k+1}^nc(x_j)\right)A(x_1x_k^{-1},\ldots, x_{k-1}x_k^{-1}) =-\frac{1}{x_1}c \left(\frac{1}{x_1}\right)c
\left(\frac{x_2}{x_1}\right)\cdots c
\left(\frac{x_n}{x_1}\right).
\end{align}
\end{lemma}
\begin{proof}
The proof is a straightforward, but unelightening, induction. We
include the computations for the convenience of the reader. We first prove  \eqref{eq:fact-A} The base case $n=1$ is easy to check:
\begin{align*}
B(x_1)-d
  \left(\frac{1}{qz}\right)^{-1}A(x_1z^{-1})&=-\frac{1}{qx_1}+\frac{1-\frac{1}{qz}}{1-\frac{1}{q}}\frac{1-\frac{1}{q}}{1-\frac{x_1}{z}}=-\frac{1}{qx_1}+\frac{1-\frac{1}{qz}}{1-\frac{x_1}{z}}=\frac{1-\frac{1}{qx_1}}{1-\frac{x_1}{z}}=-\frac{1}{x_1}c \left(\frac{1}{x_1}\right)\frac{1-x_1^{-1}}{z^{-1}-x_1^{-1}}.
\end{align*}
The recursion relations give
\begin{align*}
B(x_1,\ldots,x_{n+1})-d
\left(\frac{1}{qz}\right)^{-1}A(x_1z^{-1},\ldots,
x_{n+1}z^{-1})&=\frac{1}{q}B(x_1,\ldots,x_{n})-\frac{1}{qx_{n+1}}B(x_1x_{n+1}^{-1},\ldots,
                x_nx_{n+1}^{-1})\\
  &\ -d
    \left(\frac{1}{qz}\right)^{-1}c(x_{n+1}z^{-1})A(x_1z^{-1},\ldots,x_nz^{-1})\\
  &\ -d
\left(\frac{1}{qz}\right)^{-1}d(x_{n+1}z^{-1})A(x_1x_{n+1}^{-1},\ldots,
                x_nx_{n+1}^{-1}).
\end{align*}
Using the inductive hypothesis, we have
\begin{align*}
A(x_1z^{-1},\ldots,x_nz^{-1})&=d(q^{-1}z^{-1})B(x_1,\ldots,x_n)+\frac{d(q^{-1}z^{-1})}{x_1}c(x_1^{-1})c(x_2x_1^{-1})\ldots c(x_nx_1^{-1})\frac{1-x_1^{-1}}{z^{-1}-x_1^{-1}}\\
A(x_1x_{n+1}^{-1},\ldots,x_nx_{n+1}^{-1})&=d(q^{-1}x_{n+1}^{-1})B(x_1,\ldots,x_n)+\frac{d(q^{-1}x_{n+1}^{-1})}{x_1}c(x_1^{-1})c(x_2x_1^{-1})\ldots c(x_nx_1^{-1})\frac{1-x_1^{-1}}{x_{n+1}^{-1}-x_1^{-1}}
\end{align*}
Plugging into the expression above gives
\begin{align*}
  &=B(x_1,\ldots,x_n)\left(\frac{1}{q}+d(q^{-1}x_{n+1}^{-1})\left(-\frac{1}{qx_{n+1}}+\frac{d(x_{n+1}z^{-1})}{d(q^{-1}z^{-1})}\right)-c(x_{n+1}z^{-1})\right)\\
  &\ +\frac{1}{x_1}c(x_1^{-1})c(x_2x_1^{-1})\cdots c(x_nx_1^{-1})\left(d(q^{-1}x_{n+1}^{-1})\left(-\frac{1}{qx_{n+1}}+\frac{d(x_{n+1}z^{-1})}{d(q^{-1}z^{-1})}\right)\frac{1-x_1^{-1}}{x_{n+1}^{-1}-x_1^{-1}}-c(x_{n+1}z^{-1})\frac{1-x_1^{-1}}{z^{-1}-x_1^{-1}}\right)
\end{align*}
It's immediate to compute that the coefficient of $B(x_1,\ldots,
x_n)$ is
\begin{align*}
\frac{1}{q}+d(q^{-1}x_{n+1}^{-1})\left(-\frac{1}{qx_{n+1}}+\frac{d(x_{n+1}z^{-1})}{d(q^{-1}z^{-1})}\right)-c(x_{n+1}z^{-1})&=0
\end{align*}
and that
\begin{align*}
d(q^{-1}x_{n+1}^{-1})\left(-\frac{1}{qx_{n+1}}+\frac{d(x_{n+1}z^{-1})}{d(q^{-1}z^{-1})}\right)\frac{1-x_1^{-1}}{x_{n+1}^{-1}-x_1^{-1}}-c(x_{n+1}z^{-1})\frac{1-x_1^{-1}}{z^{-1}-x_1^{-1}}&=-c(x_{n+1}z^{-1})\frac{1-x_1^{-1}}{z^{-1}-x_1^{-1}}
\end{align*}
which finishes the inductive step.

We now turn to  \eqref{eq:fact-B}.   Again, we proceed by induction, the base case being
\begin{align*}
B(x_1)+d(x_1)=-\frac{1}{x_1}c(x_1^{-1}).
\end{align*}
By the inductive hypothesis, it suffices to verify that
\[
B(x_1,\ldots,x_{n+1})+\sum_{k=1}^{n+1}d(x_{k})\prod_{j=k+1}^{n+1}c(x_j)A(x_1x_{k}^{-1},\ldots,x_{k-1}x_{k}^{-1})=\]
\[=c(x_{n+1}x_1^{-1})\left(B(x_1,\ldots,x_{n})+\sum_{k=1}^{n}d(x_{k})\prod_{j=k+1}^{n}c(x_j)A(x_1x_{k}^{-1},\ldots,x_{k-1}x_{k}^{-1})\right)
\]
By definition, it suffices to check
\[\frac{1}{q}B(x_1,\ldots,x_{n})-\frac{1}{qx_{n+1}}A(x_1x_{n+1}^{-1},\ldots,x_nx_{n+1}^{-1})+\sum_{k=1}^{n+1}d(x_{k})\prod_{j=k+1}^{n+1}c(x_j)A(x_1x_{k}^{-1},\ldots,x_{k-1}x_{k}^{-1})=\]
\[=c(x_{n+1}x_1^{-1})B(x_1,\ldots,x_{n})+c(x_{n+1}x_1^{-1})\sum_{k=1}^{n}d(x_{k})\prod_{j=k+1}^{n}c(x_j)A(x_1x_{k}^{-1},\ldots,x_{k-1}x_{k}^{-1})
\]
which is equivalent to
\[\left(\frac{1}{q}-c(x_{n+1}x_1^{-1})\right)B(x_1,\ldots,x_{n})+\left(-\frac{1}{qx_{n+1}}+d(x_{n+1})\right)A(x_1x_{n+1}^{-1},\ldots,x_nx_{n+1}^{-1})=\]
\[=\left(c(x_{n+1}x_1^{-1})-c(x_{n+1})\right)\sum_{k=1}^{n-1}d(x_k)\prod_{j=k+1}^{n-1}A(x_1x_k^{-1},\ldots,
x_{k-1}x_k^{-1}).\]
A quick computation and another application of the inductive
hypothesis leads to verifying
\[-d(x_{n+1}x_1^{-1})B(x_1,\ldots,x_{n})-\frac{1}{x_{n+1}}c(x_{n+1}^{-1})A(x_1x_{n+1}^{-1},\ldots,x_nx_{n+1}^{-1})=\]
\[=\left(c(x_{n+1}x_1^{-1})-c(x_{n+1})\right)\left(-\frac{1}{x_1}c(x_1^{-1})c(x_2x_1^{-1})\cdots
c(x_nx_1^{-1})-B(x_1,\ldots,x_n)\right)\]
which is equivalent to
\[d(x_{n+1})B(x_1,\ldots,x_n)+\frac{1}{x_{n+1}}c(x_{n+1}^{-1})A(x_1x_{n+1}^{-1},\ldots,x_nx_{n+1}^{-1})=\frac{1}{x_1}c(x_1^{-1})c(x_2x_1^{-1})\cdots
c(x_nx_1^{-1})\frac{d(x_{n+1}^{-1})d(x_{n+1}x_1^{-1})}{d(x_1^{-1})},\]
which follows from the first part.
\end{proof}

We are now in a position to state the main factorization result:
\begin{proposition}\label{p:hard-case}
We have
\begin{align*}
\mathcal{C}_{\theta,(i,n)(n+i,2n)}&=c_{n,i}(\theta)c_{2n,n+i}(\theta)\prod_{i<j<n}c_{n,j}(\theta)c_{2n,n+j}(\theta)c_{j,i}(\theta)c_{n+j,n+i}(\theta)\\
  \mathcal{C}_{\theta,(n,2n)}&=-\frac{\theta_n(\varpi)}{\theta_{2n}(\varpi)}c_{n,2n}(\theta)\prod_{i=1}^{n-1}c_{n,n+i}(\theta)c_{n+i,n}(\theta).
\end{align*}
\end{proposition}

\begin{proof}
By definition of the Shalika local coefficient,
\begin{align*}
\mathcal{C}_{\theta,w}=\mathcal{C}_{\theta,w} q^{-\delta n}\mathcal{S}_{\theta^{w^{-1}}}f_{Bw_0J}^{\theta^{w^{-1}}}(t^{\delta})=q^{-\delta n}\mathcal{S}_\theta M_{(n,2n)} f_{Bw_0J}^{\theta^{w^{-1}}}(t^{-\delta}\varpi^\delta ) 
\end{align*}

The first part of the proposition follows immediately from Proposition \ref{p:complicated-intertwining}.
For the second part, note that 
\begin{align*}
\mathcal{C}_{\theta,(n,2n)}=q^{-\delta n}\mathcal{S}_\theta M_{(n,2n)} f_{Bw_0J}^{\theta^{(n,2n)}}(t^{-\delta}\varpi^\delta ) 
\end{align*}
By Proposition \ref{p:complicated-intertwining}, this equals
\begin{align*}
\mathcal{C}_{\theta,(n,2n)}=\prod_{i=1}^{n-1}c_{n+i,n}(\theta) \left(q^{-\delta n}\mathcal{S}_\theta(F_n)(t^{-\delta}\varpi^\delta )+\sum_{k=0}^{n-1}d_{2n,n+k}\prod_{i=k+1}^{n-1}c_{2n,n+i}(\theta)q^{-\delta n}\mathcal{S}_\theta(F_k)(t^{-\delta}\varpi^\delta )\right)
\end{align*}
The computations of \S\ref{ss:shalikaFk} and \S\ref{ss:shalikaFn}   give
\begin{align*}
 & q^{-\delta n}\mathcal{S}_\theta(F_n)(t^{-\delta}\varpi^\delta )+\sum_{k=0}^{n-1}d_{2n,n+k}\prod_{i=k+1}^{n-1}c_{2n,n+i}(\theta)q^{-\delta n}\mathcal{S}_\theta(F_k)(t^{-\delta}\varpi^\delta )\\
  &=\mathcal{I}_\psi(\frac{\theta_{2n}}{\theta_{n}},\ldots,\frac{\theta_{2n}}{\theta_{2n-1}}) +\sum_{k=0}^{n-1}d(\frac{\theta_{2n}}{\theta_{n+k}})\prod_{i=k+1}^{n-1}c(\frac{\theta_{2n}}{\theta_{n+i}})\mathcal{I}(\frac{\theta_{n+k}}{\theta_{n}},\ldots,\frac{\theta_{n+k}}{\theta_{n+k-1}})
\end{align*}
We now apply Lemma \ref{l:fact-AB} with $x_i =
\theta_{2n}\theta_{n+i-1}^{-1}(\varpi)$ and obtain that the above
expression equals
\[-\frac{\theta_n(\varpi)}{\theta_{2n}(\varpi)}\prod_{i=1}^{n}c(\theta_{n}\theta_{n+i}^{-1}(\varpi)),\]
and the desired result follows.
\end{proof}

\begin{remark}
One might ask about the factorization of the Shalika local
coefficient $\mathcal{C}_{\theta,(i,n+i)}$, which we have not
computed. The complexity of the formulas involved are much
greater. For instance, while
$M_{(n,2n)}f_{Bw_0J}^{\theta^{(n,2n)}}$ is supported on $n+1$
parahoric cells, as in the proof above, we have
\begin{align*}
  M_{(i,n+i)}f_{Bw_0J}^{\theta^{(i,n+i)}}&=\prod_{u=1}^{i-1}c_{n+u,i}(\theta)\prod_{v=i+1}^nc_{n+i,v}(\theta)\left(d_{n+i,i}(\theta)\prod_{u=1}^{i-1}c_{n+i,n+u}(\theta)\prod_{v=i+1}^nc_{i,v}(\theta)f_{Bw_0J}^\theta\right.\\
  &\hspace{1cm}+\left.\sum_{1\leqslant k\leqslant i\leqslant \ell\leqslant n}d_{\ell,i}d_{n+i,n+k}\prod_{u=i+1}^{\ell-1}c_{u,i}(\theta)\prod_{v=k+1}^{i-1}c_{n+i,n+v}(\theta)f_{B(\ell,\ell+1,\ldots, k-1,k)w_0J}^\theta\right),
\end{align*}
with the convention $d_{k,k}(\theta)=1$. This function is supported on $i(n+1-i)+1$ cells, and the computation becomes daunting.
\end{remark}

\begin{remark}\label{r:wedge2}
The coefficient $\mathcal{C}_{\theta,(n,2n)}$ can
be written as
\[\mathcal{C}_{\theta,(n,2n)}=\gamma(\theta_n/\theta_{2n},1)\prod_{i=1}^{n-1}\gamma(\theta_n/\theta_{n+i},1)\gamma(\theta_{n+i}/\theta_{n},1),\]
realizing the Shalika local coefficient as a subproduct of $\gamma(\wedge^2 \pi,1)$. 
\end{remark}

\subsection{Proof of Theorem \ref{t:shalika-sub}}\label{s:shalika-sub}
We are now ready to construct an explicit Shalika functional on parahoric spherical representations by analytic continuation. The idea is the following: restrict the Ash--Ginzburg Shalika integral to $\pi\subset \Ind_B^G\theta$. If the restriction does not vanish, we are done; if it does, factor out the Shalika local coefficients computed above. 

Suppose $\pi$ is a  parahoric spherical regular generic   representation of $G$ admitting a Shalika model. By
\cite[Cor. 1.1]{matringe:shalika} and Proposition
\ref{p:parahoric-rank2}, $\pi$ is the normalized parabolic induction
\[\pi\cong\Ind_P^G \left(\prod_{i=1}^a (\theta_i\times\theta_i^{-1})\times \prod_{i=1}^b \varepsilon_i\St \times\prod_{i=1}^c
(\eta_i\St\times\eta_i^{-1}\St)\right),\]
where $\theta_i,\varepsilon_i,\eta_i$ are unramified characters, $\St$ is the Steinberg representation on $\GL_2$, $\varepsilon_i^2=1$, and $a+b+2c=n$. 
By regularity $b\leqslant 2$, as the only possible $\varepsilon_i$ that can appear are $1$ and the  unramified quadratic character. 

Let $\theta$ be the unramified character of $T$ such that 
$\theta_{a+i}=\varepsilon_i|\cdot|^{\frac{1}{2}}$ for $1\leqslant i\leqslant b$, and
$\theta_{a+b+2i-1}=\eta_i|\cdot|^{\frac{1}{2}},
\theta_{a+b+2i}=\eta_i^{-1}|\cdot|^{\frac{1}{2}}$ for $1\leqslant i\leqslant c$; 
for $1\leqslant i\leqslant n$, set $\theta_{n+i}=\theta_i^{-1}$. The
assumption that $\pi$ is regular translates to
$\theta_i\neq\theta_j$ for all $i<j$. Then $\pi\subset\Ind_B^G\theta$ is the unique irreducible subrepresentation, by the Bernstein-Zelevinsky classification. We remark that the $U$-eigenvalue $\alpha$ is uniquely determined by $\theta$, and each spin refinement $\widetilde{\pi}$ can be obtained in this way.
Let $\tau\in W$ be given by the permutation
\begin{equation}
\tau=\prod_{i=1}^b(a+i,n+a+i)\prod_{i=1}^c(a+b+2i-1,n+a+b+2i-1)(a+b+2i,n+a+b+2i).
\end{equation}
The intertwining operator $M_\tau:I(\theta^{\tau^{-1}})\to I(\theta)$ has image $\pi$ as, on each $\GL_2$ Levi
block, the intertwining $M_{(i,j)}:\Ind_{B_2}^{\GL_2}(\eta_i|\cdot|^{-\frac{1}{2}} \times\eta_i|\cdot|^{\frac{1}{2}})\to
\Ind_{B_2}^{\GL_2}(\eta_i|\cdot|^{\frac{1}{2}} \times\eta_i|\cdot|^{-\frac{1}{2}})$ has image $\eta_i\St$. 

For a complex $2n$-tuple $s=(s_1,s_2,\ldots,s_{2n})$, consider $\theta_s= (|\cdot|^{s_1}\theta_1, \dots, |\cdot|^{s_{2n}}\theta_{2n})$. For $s$ in a Zariski open in $\mathbb{C}^{2n}$, the representation $I(\theta_s)$ is irreducible and $M_\tau:I(\theta_s^{\tau^{-1}})\to I(\theta_s)$ is an isomorphism. Let $\mathcal{F}(\theta_s)$ be the vector space of linear combinations of flat families in $I(\theta_s)$, with coefficients which are holomorphic functions in $s$. Note that $G$ acts on $\mathcal{F}(\theta_s)$ by admissibility, and $M_\tau:\mathcal{F}(\theta_s^{\tau^{-1}})\to \mathcal{F}(\theta_s)$ is an isomorphism. 

For $f\in \pi\subset I(\theta)$, consider $f_s\in
\mathcal{F}(\theta_s)$, e.g., a flat family, with $f=f_0$. The
naive definition $\mathcal{S}_{\AG}(f)=\lim\limits_{s\to
  0}\mathcal{S}_{\AG}^{\theta_s}(f_s)$ might not work, as it
might happen that $\mathcal{S}_{\AG}(f)$ vanishes for all $f\in
\pi$. In fact, this always happens if the monodromy rank of
$\pi$ is larger than 2. However, the computations in this
section show that
\begin{equation}\label{eq:sh-c}
\mathcal{S}_{\AG}^{\theta_s}(f_s)=\mathcal{C}_{\theta_s,\tau}\cdot
\mathcal{S}_{\AG}^{\theta_s^{\tau^{-1}}}(M_\tau^{-1}(f_s)).
\end{equation}
The function $\mathcal{S}_{\AG}^{\theta_s}(f_s)$ is holomorphic in $s$, but 
$\mathcal{S}_{\AG}^{\theta_s^{\tau^{-1}}}(M_\tau^{-1}(f_s))$
might not be. If it was, we could define
\[\mathcal{S}_\pi(f)=\lim_{s\to
  0}\frac{\mathcal{S}^{\theta_s}_{\AG}(f_s)}{\mathcal{C}_{\theta_s,\tau}}=\lim_{s\to
  0}\mathcal{S}_{\AG}^{\theta_s^{\tau^{-1}}}(M_\tau^{-1}(f_s)),\]
which we cannot show to be a Shalika functional. To account for this, we must take a preimage under the
intertwining operator before considering families. We begin with a crucial technical result:
\begin{proposition}\label{p:kerT}
The kernel of the intertwining operator $M_\tau:I(\theta^{\tau^{-1}})\to  I(\theta)$ is contained in the kernel of the Ash--Ginzburg Shalika functional $\mathcal{S}_{\AG}^{\theta^{\tau^{-1}}}:I(\theta^{\tau^{-1}})\to \mathbb{C}$.
\end{proposition}
\begin{proof}
By the Bernstein-Zelevinsky classification, since $\Im
M_\tau=\pi$, $(\ker M_\tau)^{\ss}=\bigoplus \sigma$, where
$\sigma$ runs over all components of $ I(\theta)$ different
from $\pi$. Our strategy is to realize $\ker M_\tau$ in
the linear span of a collection of auxiliary intertwining
operators.

Suppose $b\neq 0$ (otherwise this paragraph does not apply). For
each $1\leqslant i\leqslant b$, we will move the characters $\varepsilon_i|\cdot|^{\pm 1/2}$ to positions $n$ and $2n$, swap, and move back. Let
\[\nu_i = (2n,2n-1,\ldots,n+a+i)(n,n-1,\ldots, a+i)\]
and $\mu_i = \nu_i^{-1}(n,2n)\nu_i$. The semisimplification of
the image of $M_{\mu_i}: I(\theta^{\mu_i^{-1}\tau^{-1}})\to
I(\theta^{\tau^{-1}})$ will contain all components which
are $\varepsilon_i\circ\det$ on the Levi block $(a+i,n+a+i)$. Indeed,
since $\theta$ is regular, $M_{\nu_i}$ is an isomorphism, and
the image of $M_{(n,2n)}$ will give the character on the Levi
block.

By Proposition \ref{p:complicated-intertwining}, we see that
\[\mathcal{S}_{\AG}^{\theta^{\tau^{-1}}}\circ
M_{\mu_i}=\mathcal{C}_{\theta^{\tau^{-1}},\mu_i}\mathcal{S}_{\AG}^{\theta^{\mu_i^{-1}\tau^{-1}}},\]
where
$\mathcal{C}_{\theta^{\tau^{-1}},\mu_i}$ contains as factor
\[c_{n,2n}(\theta^{\nu_i^{-1}\tau^{-1}})=c\left(\frac{\varepsilon_i|\cdot|^{-1/2}}{\varepsilon_i|\cdot|^{1/2}}\right)=0.\]
We deduce that $\mathcal{S}_{\AG}^{\theta^{\tau^{-1}}}$ vanishes on $\Im M_{\mu_i}$.

Suppose, now, that $1\leqslant i\leqslant c$. In this case, we need two
auxiliary intertwining operators. Consider the permutations
$\nu_i=(2n,2n-1,\ldots, n+a+b+2i)(n,n-1,\ldots,a+b+2i)$,
$\mu_i=\nu_i^{-1}(n,2n)\nu_i$, and
$\mu'_i=\nu_i^{-1}(n,2n)\nu_i(a+b+2i-1,a+b+2i)(n+a+b+2i-1,n+a+b+2i)$.

First, the semisimplification of the image of
$M_{\mu_i}: I(\theta^{\mu_i^{-1}\tau^{-1}})\to
I(\theta^{\tau^{-1}})$ contains all components which are
$\eta\circ\det$ on the Levi block $(a+b+2i,n+a+b+2i-1)$, as the
only linked characters that swap order are $\eta|\cdot|^{\pm
  1/2}$. Again, we see that
$\mathcal{S}_{\AG}^{\theta^{\tau^{-1}}}\circ
M_{\mu_i}=\mathcal{C}_{\theta^{\tau^{-1}},\mu_i}\mathcal{S}_{\AG}^{\theta^{\mu_i^{-1}\tau^{-1}}}$,
where $\mathcal{C}_{\theta^{\tau^{-1}},\mu_i}$ contains as factor
$c_{n,2n-1}(\theta^{\nu_i^{-1}\tau^{-1}})=c\left(\frac{\eta_i|\cdot|^{-1/2}}{\eta_i|\cdot|^{1/2}}\right)=0$,
so $\mathcal{S}_{\AG}^{\theta^{\tau^{-1}}}$ vanishes on $\Im M_{\mu_i}$.

Finally, the semisimplification of the image of
$M_{\mu'_i}: I(\theta^{(\mu'_i)^{-1}\tau^{-1}})\to
I(\theta^{\tau^{-1}})$ contains all components which are
$\eta^{-1}\circ\det$ on the Levi block $(a+b+2i-1,n+a+b+2i)$, as the
only linked characters that swap order are $\eta^{-1}|\cdot|^{\pm
  1/2}$. Again, we see that
$\mathcal{S}_{\AG}^{\theta^{\tau^{-1}}}\circ
M_{\mu'_i}=\mathcal{C}_{\theta^{\tau^{-1}},\mu'_i}\mathcal{S}_{\AG}^{\theta^{(\mu'_i)^{-1}\tau^{-1}}}$,
where
$\mathcal{C}_{\theta^{\tau^{-1}},\mu'_i}$ contains as factor
$c_{n,2n-1}(\theta^{(a+b+2i-1,a+b+2i)(n+a+b+2i-1,n+a+b+2i)\nu_i^{-1}\tau^{-1}})=c(\frac{\eta_i^{-1}|\cdot|^{-1/2}}{\eta_i^{-1}|\cdot|^{1/2}})=0$,
so $\mathcal{S}_{\AG}^{\theta^{\tau^{-1}}}$ vanishes on $\Im
M_{\mu'_i}$.

We conclude that $\mathcal{S}_{\AG}^{\theta^{\tau^{-1}}}$ vanishes on the linear span of all components of $I(\theta^{\tau^{-1}})$ other than $\pi$, and thus also on $\ker M_\tau$, as desired.
\end{proof}

\begin{definition}
For $f\in \pi$ let $g\in I(\theta^{\tau^{-1}})$ such that
$f=M_\tau g$, and let $g_s$ any analytic family such that
$g_0=g$. Define
\[\mathcal{S}_\pi(f)=\lim_{s\to 0}\mathcal{S}_{\AG}^{\theta^{\tau^{-1}}}(g_s).\]
\end{definition}
This will satisfy the requirements of Theorem
\ref{t:shalika-sub}. First, $\mathcal{S}_\pi(f)$ is
well-defined. Indeed, if $g'$ is another such preimage and
$g'_s$ is an analytic family through $g'$ then
\[\lim_{s\to
  0}\mathcal{S}_{\AG}^{\theta^{\tau^{-1}}}(g_s)=\mathcal{S}_{\AG}^{\theta^{\tau^{-1}}}(g)=\mathcal{S}_{\AG}^{\theta^{\tau^{-1}}}(g')=\lim_{s\to
  0}\mathcal{S}_{\AG}^{\theta^{\tau^{-1}}}(g'_s),\]
as $g-g'\in \ker M_{\tau}\subset \ker \mathcal{S}_{\AG}^{\theta^{\tau^{-1}}}$ by Proposition \ref{p:kerT}. Second, $\mathcal{S}_\pi$ is a Shalika functional, being defined by the analytic continuation of Shalika functionals. Finally, 
if $f=M_\tau f_{Bw_0J}^{\theta^{\tau^{-1}}}$, then
$g_s=f_{Bw_0J}^{\theta_s^{\tau^{-1}}}$ is a flat family,
and
\begin{align*}
\mathcal{S}_\pi(f)(t^{-\delta}\varpi^\delta )&=\lim_{s\to 0}\mathcal{S}_{\AG}^{\theta_s^{\tau^{-1}}}(f_{Bw_0J}^{\theta^{\tau^{-1}}})(t^{-\delta}\varpi^\delta )=1.
\end{align*}

\section{Parahoric level $p$-adic $L$-functions and $p$-adic families}\label{s3}

We first  recall some automorphic results from \cite{DJR}, then we present the Friedberg--Jacquet style linear functional on the overconvergent coholmology constructed in \cite{BDW} while setting up global notations.

Let $F$ be a totally real number field of degree $d$, let $\cO_F$ be its ring of integers and $\Sigma$ the set of its real embeddings. Let $\A = \A_f\times \R$ denote the ring of adeles of $\Q$. For  $v$  a non-archimedean place of $F$, we let $F_v$ be the completion of $F$ at $v$, denote by $\cO_v$ the ring of integers in $F_v$, and fix a uniformiser $\varpi_v$. 

Let $n \geqslant 1$ and let $G$ denote the algebraic group $\mathrm{Res}_{\cO_F/\Z}\mathrm{GL}_{2n}$, $B= \mathrm{Res}_{\cO_F/\Z}B_{2n}$ for the Borel subgroup of upper triangular matrices, $N$ its the unipotent radical, and $T= \mathrm{Res}_{\cO_F/\Z}T_{2n}$ be the maximal split torus of diagonal matrices.  
Let $H=\mathrm{Res}_{\cO_F/\Z}(\GL_n \times \GL_n)$ which we diagonally embed into   $G$. 
Consider the  Siegel parabolic subgroup $Q$ of $G$  whose Levi subgroup is $H$, and let $U$ be its  unipotent radical.

Fix a rational prime $p$ and an embedding $\iota_p : \overline{\Q} \hookrightarrow \overline{\Q}_p$. It determines partition of $\Sigma$ into subsets 
$\Sigma_\mathfrak{p}$ indexed by the set  $\Sigma_p$ of primes $\mathfrak{p}$ of $F$ above $p$.

\subsection{The automorphic setup} \label{sec:setup} 
Let $\pi$ be a cuspidal automorphic representations of $\GL_{2n}(\A_F)$ having central character $\omega$.
We assume that $\pi$  is regular algebraic, meaning that it is cohomological of weight
$\lambda= (\lambda_{\sigma})_{\sigma \in \Sigma}$ which is dominant, integral and pure 
{\it i.e.} there exists $\sw  \in \Z$, the purity weight, such that  for all $\sigma \in \Sigma$,  $\lambda_{\sigma, 1}\geqslant  \dots \geqslant \lambda_{\sigma, 2n}$ are in $\Z$ and  $\lambda_{\sigma, i}+ \lambda_{\sigma, 2n- i+ 1}= \sw$   for all  $1 \leqslant  i \leqslant  2n$. 

We further assume that $\pi$ is essentially self-dual of symplectic type with respect to a Hecke character $\eta$ or, equivalently by Asgari--Shahidi~\cite{AS06}, that it is a functorial transfer of a globally generic representation of $\mathrm{GSpin}_{2n+1}$ having central character $\eta$, 
where we recall that $\eta^n=\omega$ and that $\eta |\cdot |^{\sw}$ has finite order. 
 Of utmost importance to us will be a third equivalent characterization, namely that $\pi$ admits a Shalika model with respect to $\eta$. 
 We call such a $\pi$ a RASCAR.  We also recall that this implies that for each place $v$ of $F$,  $\pi_v$ admits a Shalika model with respect to $\eta_v$. 

In this paper, we will assume that $\pi$ is parahoric spherical,  {\it i.e.}  $\pi_v^{J_v}\ne \{0\} $ for each finite place $v$ of $F$, where the parahoric subgroup $J_v$ of $\GL_{2n}(\cO_v)$ consists of matrices whose reduction modulo $\varpi_v$ belongs to $Q$. This implies that $\eta_v$ is unramified, hence there exists an unramified character $\xi_v$ such that $\xi_v^2=\eta_v^{-1}$ and the results from \S\ref{s1}-\S\ref{s:intertwining} apply to  $\pi_v\otimes \xi_v$  which admits  a Shalika model with respect to the trivial character.  

As recalled in \S\ref{sec:local},  $\pi_v$ contributes to the (unitarily normalized)  parabolic induction $\Ind_{B(F_v)}^{G(F_v)}\theta_v$ of an unramified character $\theta_v$ and we let $\alpha_{v,i}=\theta_{v,i}(\varpi_v)$ denote the Satake parameters. As $\pi_v\simeq\pi_v^\vee\otimes \eta_v$, the Satake parameters 
can be paired up with product $\eta(\varpi_v)$ in each couple.  

Let $S=\{v\nmid p\infty: \pi_v\text{ not spherical}\}$  be the set of bad places for $\pi$. 
Both  $p$-adic deformations and  constructions of $p$-adic $L$-functions  require a preliminary step consisting in the choice of  $p$-adic $Q$-refinement $\widetilde{\pi}=\left(\pi,(\alpha_v)_{v\in S\cup S_p}\right)$ of $\pi$, where $\alpha_v$ is a $U_v$-eigenvalue on $\pi_v^{J_v}$, and it  is essential to require the refinement $\widetilde{\pi}_v=(\pi_v,\alpha_v)$ to be spin (see \cite[\S6]{BDGJW}). By \cite{DJR} one has 
$\alpha_v=q_v^{n^2/2}\prod\limits_{i\in I}\alpha_{v,i}$, where $q_v=\mathrm{N}_{F/\Q}(v)$ and 
$I$ is a subset of  $\{1,\dots, 2n\}$  of cardinality $n$, and being spin amounts to choosing in $I$ exactly one 
Satake parameter per couple with product $\eta(\varpi_v)$. 

We will always assume the refinement $\widetilde{\pi}_v=(\pi_v,\alpha_v)$  to be regular, {\it i.e.} that  the generalized eigenspace $\pi_v^{J_v} \lsem U_v-\alpha_v\rsem$ is a line. As the refinement is spin this already implies that $\alpha_{v,i}\ne \alpha_{v,j}$ for all $i\in I, j\notin I$, but for simplicity we will 
further assume $\theta_v$ to be regular,  {\it i.e.} the $\alpha_{v,i}$ to be pairwise distinct. 

The explicit Shalika functionals for unramified principal series  constructed by Ash--Ginzburg \cite{AG}  play a central role in our approach. 
They specifically require the induced character to be ordered so that: 
\[\theta_{v,i}\theta_{v,n+i}=\eta \text{ for all }  1 \leqslant  i \leqslant n.  \]
Henceforth we will consider our parahoric spherical $\pi_v$ as a subrepresentation of $\Ind_{B(F_v)}^{G(F_v)}\theta_v$ with 
$\alpha_{v,i}\cdot \alpha_{v,n+i}=\eta(\varpi_v)$ and $\alpha_{v}=q_v^{n^2/2}\prod\limits_{n+1}^{2n}\alpha_{v,i}$.

\subsection{Automorphic $p$-adic $L$-functions}  \label{sec:padicL}




Let $\widetilde\pi^S = (\pi, (\alpha_v)_{v\in S\cup S_p})$ be a  regularly $Q$-refined RACAR
of weight $\lambda$ as in \S\ref{sec:setup}.  Assume further that $\pi$ admits a  $(\eta,\psi)$-Shalika model:
\begin{equation}\label{eq:shalika integral}
\mathcal{S}_{\psi}^\eta: \pi\hookrightarrow \Ind_{S(\A)}^{G(\A)}(\eta\otimes\psi), \qquad \varphi \mapsto 	\mathcal{S}_{\psi}^\eta(\varphi): g \mapsto \int_{Z_G(\A)\mathcal{S}(\Q)\backslash\mathcal{S}(\A)} \varphi(sg) \ (\eta \otimes \psi)^{-1}(s)ds.
\end{equation}

Consider the open compact subgroup $K(\widetilde\pi^S)=\prod_v K_v$ of  $G(\A_f)$ where
$K_v=\GL_{2n}(\cO_v)$  is  the maximal hyperspecial subgroup for $v\notin S\cup S_p$ and 
$K_v=J_v$  is  the standard $Q$-parahoric subgroup for $v\in S\cup S_p$.  
Let $\cH^S$ denote the spherical Hecke algebra away from $S\cup S_p$ and let $\fm_\pi$ be its maximal ideal attached to $\pi$
(see \cite[Def.~2.1]{BDW}, where this commutative algebra is denoted by  $\cH'$). 
Consider  the following maximal ideal of $\cH^S[U_v - \alpha_v, v\in S\cup S_p]$: 
\[\fm_{\widetilde\pi}^S=(\fm_\pi, U_v - \alpha_v, v\in S\cup S_p).\]

Given a character $\epsilon$ of  $\{\pm1\}^\Sigma$, we will now endow  the line 
$\pi_f^{K(\widetilde\pi)}[ U_v - \alpha_v, v\in S\cup S_p]$ with two rational structures 
 allowing us to define Betti--Shalika periods measuring the ratio between the two.

The locally symmetric space $S_K = G(\Q)\backslash G(\A)/KK_{\infty}^{\circ}$  is a  $t=d(2n-1)(n+1)$-dimensional real orbifold. 
By \cite[Prop.~2.3]{BDW}, there is a Hecke-equivariant isomorphism
	\begin{equation}\label{eq:betti-shalika}
		\pi_f^{K(\widetilde\pi^S)}[ U_v - \alpha_v, v\in S\cup S_p] \xrightarrow{\sim} \mathrm{H}_c^t(S_{K(\widetilde\pi^S)}, \mathscr{V}_{\lambda}^\vee(\overline{\Q}_p))_{\fm_{\widetilde\pi}^S}^\epsilon.
	\end{equation}
depending on the choice of basis of the relative Lie algebra cohomology  and  $\iota_p$. The right hand side is a line admitting a  basis $\phi_{\widetilde\pi^S}^\epsilon$   defined over the field of rationality of $\widetilde\pi^S$ (that is, the field of rationality of $\pi$ to which we adjoin the $\alpha_v$'s) and which can be further assumed to be $p$-integral.

We will now use the Shalika model $\mathcal{S}_{\psi_f}^{\eta_f}$ of $\pi_f$ to fix an appropriate  basis of the left hand side. 

For $v\notin S\cup S_p\cup \Sigma$, Friedberg--Jacquet \cite{friedberg-jacquet} showed the existence of a  spherical vector $W^{\mathrm{FJ}}_v \in \mathcal{S}_{\psi_v}^{\eta_v}(\pi_v)$  such that $W^{\mathrm{FJ}}_v(t_v^{-\delta_v})=1$. Moreover they show that 
for all unramified quasi-characters $\chi_v : F_v^\times \to \C^\times$, we have
\begin{equation}\label{eq:jacquet-friedberg test vector}
	\zeta_v\left(s+\tfrac{1}{2}, W^{\mathrm{FJ}}_v, \chi_v\right) = [q_v^s\chi_v(\varpi_v)]^{n\delta_v}\cdot L\left(\pi_v \otimes \chi_v,s+\tfrac{1}{2}\right). 
\end{equation}

For $v\in S\cup S_p$, by Theorem~\ref{t:shalika-sub} there exists $W_{\widetilde\pi_v}\in \mathcal{S}_{\psi_v}^{\eta_v}(\pi_v^{J_v})[ U_v-\alpha_v]$ such that  $W_{\widetilde\pi_v}(t_v^{-\delta_v})=1$. Moreover by 	 Proposition~\ref{p:zeta-local-constants} for all unramified quasi-characters $\chi_v : F_v^\times \to \C^\times$,
 we have
 \begin{align*}
\zeta(s+\tfrac{1}{2},W_{\widetilde\pi_v},\chi_v)&=q_v^{\delta_v \left[(s+\frac{1}{2})n-\frac{n^2}{2}\right]}\left(1-\frac{1}{q_v}\right)^n\chi_v(\varpi_v)^{-\delta_v n}L(\theta'_{n+i,v}\chi_v,s+\tfrac{1}{2}).
\end{align*}

A straightforward consequence of this  is the existence of an explicit  Friedberg--Jacquet test vector
\[W_{\pi_v}=\sum\limits_{\alpha_v} \kappa_{\alpha_v} W_{\widetilde\pi_v}.\]
for $\pi_v$ parahoric spherical, where the sum is over all possible refinements $\widetilde\pi_v=(\pi_v,\alpha_v)$ of $\pi_v$ (assumed all regular) and 
the rational constants $\kappa_{\alpha_v}$ can be explicitly computed using partial fractions decompositions.

Thus  $\left(\mathcal{S}_{\psi_f}^{\eta_f}\right)^{-1}\left(\bigotimes_{v\notin S\cup S_p\cup \Sigma} W^{\mathrm{FJ}}_v 
\bigotimes_{v\in  S} W_{\pi_v} \bigotimes_{v\in  S_p} W_{\widetilde\pi_v}\right)$ is a   basis of the left hand side of \eqref{eq:betti-shalika} and we define the Betti--Shalika period
$\Omega_{\widetilde\pi}^\epsilon$ as its coordinate in the basis $\phi_{\widetilde\pi}^\epsilon=
\sum\limits_{\alpha_S} \kappa_{\alpha_S} \phi_{\widetilde\pi^S}^\epsilon$ that we have  chosen earlier.

 Assuming  that $\widetilde\pi$ is non-$Q$-critical in the sense of \cite[Def.~3.14]{BDW}, {\it i.e.},  that the natural map on   generalized eigenspaces 
\[	\mathrm{H}_c^{\bullet}(S_{K(\widetilde\pi^S)}, \mathscr{D}_\lambda(L))_{\fm_{\widetilde\pi}^S}  
	\xrightarrow{\sim}  \mathrm{H}_c^{\bullet}(S_{K(\widetilde\pi^S)},\mathscr{V}_\lambda^\vee(L))_{\fm_{\widetilde\pi}^S}
\]
	is an isomorphism, allows us  for each $\epsilon \in \{\pm1\}^\Sigma$ to uniquely lift  $\phi_{\widetilde\pi^S}^\epsilon$  to an $U_p$-eigenclass 
	$\Phi_{\widetilde\pi^S}^\epsilon \in \mathrm{H}_c^t(S_{K(\widetilde\pi^S)},\mathscr{D}_\lambda)^\epsilon_{\fm_{\widetilde\pi}}$. 
	Recall that the $U_p$-eigenvalue  on $\Phi_{\widetilde\pi^S}^\epsilon$ equals $\alpha_p^{\circ} =  \lambda(t_p)\alpha_p$. Let $\Phi_{\widetilde\pi}^\epsilon = \sum\limits_{\alpha_S} \kappa_{\alpha_S} \Phi_{\widetilde\pi^S}^\epsilon$.

	Let $\cL_p(\tilde\pi) = A^{-1}\cdot  \mu^{\eta_0}(\Phi_{\tilde\pi})$ be the $L$-valued distribution on $\Gal_p$ attached to  $\Phi_{\tilde\pi} = \sum\limits_{\epsilon \in \{\pm 1\}^\Sigma} \Phi_{\tilde\pi}^\epsilon$ constructed in \cite[Def.~6.17]{BDW}, where $A$ is the global constant from \cite[Def.~6.23]{BDW}.

\begin{theorem} \label{thm:non-ordinary}
The distribution $\cL_p({\tilde{\pi}})$ is admissible of growth $h_p = v_p(\alpha_p^{\circ})$. For every finite order Hecke character $\chi$ of $F$ of conductor $p^{\beta}$, and all $j \in \mathrm{Crit}(\lambda)$, we have
	\begin{align}\label{eq:interpolation}
		\iota_p^{-1}(\cL_p(\tilde\pi, \chi\chi_{\cyc}^j)) =  \mathcal{G}(\chi_f)^n  \mathrm{N}_{F/\Q}(-i\mathfrak{d})^{jn}  \prod_{\fp\in \Sigma_p} e_{\fp}(\tilde\pi_{\fp},\chi_{\fp},j)
\cdot \frac{L^{(p)}\big(\pi\otimes\chi, j+\tfrac{1}{2}\big)}{\Omega_{\tilde\pi}^{\epsilon}},
	\end{align}
	where  $\epsilon = (\chi\chi_{\cyc}^j\eta)_\infty$,  $\mathcal{G}(\chi_f)$ is the Gauss sum and  
			\[
	e_{\fp}(\tilde\pi_{\fp},\chi_{\fp},j) = \left\{\begin{array}{cl}\left(q_{\fp}^{nj + \binom{n}{2}}\alpha_{\fp}^{-1}\right)^{\beta_{\fp}} &: \chi_{\fp} \text{ ramified},\\
			\prod\limits_{i=n+1}^{2n}
			\frac{1-(\theta'_{\fp,i}\chi_{\fp})^{-1}(\varpi_{\fp})q_{\fp}^{j-1/2}}{1-\theta'_{\fp,i}\chi_{\fp}(\varpi_{\fp})q_{\fp}^{-j-1/2}}. &: \chi_{\fp} \text{ unramified.}\end{array}\right.
		\]
\end{theorem}

\subsection{The eigenvariety and families of $p$-adic $L$-functions} 
Here we follow closely \cite[\S8.3]{BDW}. 
Let $\pi$ be a RASCAR of weight $\lambda$  which is parahoric spherical at every finite place and recall that $S$ denotes the set of bad places outside $p$. 
Let $\widetilde\pi^S$ be a (regular) non-$Q$-critically refinement as in \S\ref{sec:padicL}.  
Let $\Omega$ be a sufficiently small affinoid neigbourhood of $\lambda$ in the parahoric weight space $\mathcal{W}^Q$ such that 
the overconvergent cohomology $\mathrm{H}_c^t(S_{(\widetilde\pi^S)},\mathcal{D}_\Omega)$ admit slope decomposition with respect to the compact operator 
$U_p$. It is endowed with an $\mathcal{O}(\Omega)$-linear  action of the Hecke algebra $\cH^S[U_v - \alpha_v, v\in S\cup S_p]$ yielding for each $h\in \Q_{\geqslant 0}$ a local piece $\mathcal{E}^S_{\Omega,h}$ of an eigenvariety $\mathcal{E}^S$. While this eigenvariety slightly differs from the one considered in 
{\it loc. cit. } in that, locally at $v\in S$ we have put the Hecke operator $U_v$  instead of diamond like operators $S_v$, all the constructions and arguments transfer {\it mutatis mutandis } yielding the following theorem. Indeed, the cyclicity  follows from the fact that we use top degree cohomology, while the \'etaleness follows from the local multiplicity for the family  ensured by \eqref{eq:betti-shalika} being a line. The proof itself is a subtle interplay between the closeness of the condition of being Spin and the openness of the condition of being Shalika --  a major theme in  {\it loc. cit. } to which we refer to more details. 

\begin{theorem} \label{t:eigenvariety}
Suppose that $\lambda$ is regular and that $\tilde\pi^S$ is a regular,  strongly non-$Q$-critical refinement in the sense of \cite[Def.~3.14]{BDW}. 
	 Then, after possibly shrinking $\Omega$, 
	the weight map $\mathsf{w}:\mathcal{E}^S\to \mathcal{W}^Q$ is \'etale at the point 	$x_{\tilde\pi^S}$ attached to $\tilde\pi^S$, i.e.,  
there is a 	 connected component $\mathcal{C}$  of $\mathcal{E}^S$ through $x_{\tilde\pi^S}$ such  that $\mathsf{w}:\mathcal{C}\xrightarrow{\sim} \Omega$. 
	
Moreover, for each $\epsilon \in \{\pm1\}^\Sigma$, there exists a Hecke eigenclass $\Phi_{\mathcal{C}}^\epsilon \in 
\mathrm{H}_c^t(S_{K(\widetilde\pi^S)},\mathcal{D}_\Omega)^{\epsilon}$ such that for a  Zariski dense set $\mathcal{C}^{\mathrm{cl}}$  of $x\in \mathcal{C}$  corresponding to non-$Q$-critical refined parahoric spherical RASCARs $\tilde\pi_x^S$,  the specialization of $\Phi_{\mathcal{C}}^\epsilon$ at $\mathsf{w}(x)$ generates 
 $\mathrm{H}_c^t(S_{K(\widetilde\pi^S)},\mathcal{D}_{\mathsf{w}(x)})^{\epsilon}$. 
 \end{theorem} 

 The above theorem implies existance of $p$-adic periods $c_x^{\epsilon}\in L^\times$ such that the  specialization of $\Phi_{\mathcal{C}}^\epsilon$ at $\mathsf{w}(x)$ equals $c_x^{\epsilon}\cdot \Phi_{\tilde\pi_x^S}^\epsilon$. It also allows to define a $\mathcal{O}(\Omega)$-valued distribution 
 $\mathcal{L}_p^{\mathcal{C},\epsilon} = A^{-1}\cdot  \mu^{\eta_0}(\Phi_{\mathcal{C}}^\epsilon)$ on $\Gal_p$ via \cite[Def.~6.17]{BDW}, 
 whose specialization at any  $x\in  \mathcal{C}^{\mathrm{cl}}$ equals $c_x^{\epsilon}\cdot \cL^\epsilon_p(\tilde{\pi}_x^S)$. 
The crucial point to check is the existence of  $\Omega$  such that  $\tilde\pi_x^S$ is  regular  for all $x\in \mathcal{C}^{\mathrm{cl}}$ so that 
Theorem~\ref{thm:non-ordinary} requiring regularity of the induced character at all  $v\in S\cup S_p$ applies. 
This is ensured by the regularity of $\pi_v^S$ via the following Galois theoretic argument. There exists a $\mathcal{O}_\Omega$-valued pseudo-representation of $\mathrm{Gal}_F$ (or more precisely a Chenevier determinant) interpolating the Galois representations $\rho_x$ attached to  $x\in  \mathcal{C}^{\mathrm{cl}}$ (when $\rho_\pi$ is irreducible there is even a $\mathcal{O}_\Omega$-valued representation). 
By the Local-Global Compatibility, which is  known in our case, the Satake parameters of the inducing character of $\pi_{x,v}$ can be read off the Weil--Deligne representation attached to the restriction of $\rho_x$ at $v$. The continuity of the pseudo-representation and the 
 regularity of $\pi_v$ then ensures the claim. 

In closing, let us observe that while the $p$-adic $L$-function   $\cL^\epsilon_p(\tilde{\pi}^S)$ is not uniquely determined by its interpolation property 
  for  $\tilde\pi^S$ of critical slope (but still non-$Q$-critical), the $p$-adic $L$-function for the family  $\mathcal{L}_p^{\mathcal{C},\epsilon}$ is 
  always uniquely determined by  interpolation because of the density of the  non-critical slope classical points.


\begin{thebibliography}{10}


\bibitem{AS06}
{\sc M.~Asgari and F.~Shahidi}, {\em Generic transfer for general spin groups},
  Duke Math. J., 132 (2006), pp.~137--190.

\bibitem{AG}
{\sc A.~Ash and D.~Ginzburg}, {\em {$p$}-adic {$L$}-functions for {${\rm
  GL}(2n)$}}, Invent. Math., 116 (1994), pp.~27--73.

\bibitem{BDGJW}
{\sc D.~Barrera, M.~Dimitrov, A.~Graham, A.~Jorza and C.~Williams},
  {\em On the ${\rm GL}(2n)$ eigenvariety: branching laws, {S}halika families and
  $p$-adic ${L}$-functions}, J. Assoc.\ Math.\ Res.
\newblock To appear.


\bibitem{BDW}
{\sc D.~Barrera, M.~Dimitrov and C.~Williams}, {\em On $p$-adic
  $L$-functions for $\operatorname{GL}_{2n}$ in finite slope Shalika families}, 2021.

\bibitem{BGW}
{\sc D.~Barrera, A.~Graham and C.~Williams}, {\em On $p$-refined
  {F}riedberg--{J}acquet integrals and the classical symplectic locus in the
  {$\mathrm{GL}_{2n}$}-eigenvariety}, Res.\ Number Theory, 11 (2025),
  pp.~1--57.


\bibitem{casselman:unramified}
{\sc W.~Casselman}, {\em The unramified principal series of
  {$\mathfrak{p}$}-adic groups. {I}. {T}he spherical function}, Compositio
  Math., 40 (1980), pp.~387--406.

\bibitem{casselman:book}
\leavevmode\vrule height 2pt depth -1.6pt width 23pt, {\em Introduction to the
  theory of admissible representations of $p$-adic reductive groups}.
\newblock 1995.

\bibitem{coates:motivic-Lp}
{\sc J.~Coates}, {\em Motivic {$p$}-adic {$L$}-functions}, in {$L$}-functions
  and arithmetic ({D}urham, 1989), vol.~153 of London Math. Soc. Lecture Note
  Ser., Cambridge Univ. Press, Cambridge, 1991, pp.~141--172.

\bibitem{DJR}
{\sc M.~Dimitrov, F.~Januszewski, and A.~Raghuram}, {\em {$L$}-functions of
  {$\mathrm{GL}_{2n}$}: {$p$}-adic properties and non-vanishing of twists},
  Compos. Math., 156 (2020), pp.~2437--2468.

\bibitem{friedberg-jacquet}
{\sc S.~Friedberg and H.~Jacquet}, {\em Linear periods}, J. Reine Angew. Math.,
  443 (1993), pp.~91--139.

\bibitem{hida:gl2gl2}
{\sc H.~Hida}, {\em On {$p$}-adic {$L$}-functions of {${\rm GL}(2)\times {\rm
  GL}(2)$} over totally real fields}, Ann. Inst. Fourier (Grenoble), 41 (1991),
  pp.~311--391.

\bibitem{jacquet-rallis}
{\sc H.~Jacquet and S.~Rallis}, {\em Uniqueness of linear periods}, Compositio
  Math., 102 (1996), pp.~65--123.

\bibitem{januszewski:rankin-selberg}
{\sc F.~Januszewski}, {\em On {$p$}-adic {$L$}-functions for
  {$\operatorname{GL}(n)\times\operatorname{GL}(n-1)$} over totally real
  fields}, Int. Math. Res. Not. IMRN,  (2015), pp.~7884--7949.

\bibitem{kim-shahidi:gl2gl3}
{\sc H.~H. Kim and F.~Shahidi}, {\em Functorial products for {${\rm
  GL}_2\times{\rm GL}_3$} and the symmetric cube for {${\rm GL}_2$}}, Ann. of
  Math. (2), 155 (2002), pp.~837--893.
\newblock With an appendix by Colin J. Bushnell and Guy Henniart.

\bibitem{loeffler-zerbes:bsd}
{\sc D.~Loeffler and S.~L. Zerbes}, {\em On the Birch--Swinnerton-Dyer
  conjecture for modular abelian surfaces}, 2021.

\bibitem{matringe:shalika}
{\sc N.~Matringe}, {\em Shalika periods and parabolic induction for {${\rm GL}(n)$}
  over a non-archimedean local field}, Bull. Lond. Math. Soc., 49 (2017),
  pp.~417--427.  
  
\bibitem{schmidt:gl2gl3}
{\sc C.-G. Schmidt}, {\em Relative modular symbols and {$p$}-adic
  {R}ankin-{S}elberg convolutions}, Invent. Math., 112 (1993), pp.~31--76.

\bibitem{shahidi:L}
{\sc F.~Shahidi}, {\em On certain {$L$}-functions}, Amer. J. Math., 103 (1981),
  pp.~297--355.

\bibitem{skinner:bsd}
{\sc C.~Skinner}, {\em A converse to a theorem of {G}ross, {Z}agier, and {K}olyvagin}, 
Ann. of Math. (2), 191 (2020), pp.~329--354. 
\end{thebibliography}
\end{document}